\documentclass[a4paper]{article}

\usepackage[top=2.25cm, bottom=2.25cm, left=2.0cm, right=2.0cm]{geometry}
\usepackage{amsfonts}
\usepackage{latexsym}
\usepackage{amsmath}
\usepackage{setspace}
\usepackage{comment}
\usepackage{bm}
\usepackage{graphicx}
\usepackage{subcaption}
\usepackage{epstopdf}
\usepackage{color}
\usepackage{algorithmicx}
\usepackage{algorithm}
\usepackage{algpseudocode}
\usepackage{amsopn}
\usepackage{amssymb}
\usepackage{mathtools}
\usepackage{amsthm}
\usepackage{booktabs}
\usepackage{placeins}
\usepackage{siunitx}

\newcommand{\mini}{\mathop{\mbox{minimize}}}

\newcommand{\subj}{\mathop{\mbox{subject to}}}

\newcommand{\R}{\mathbb{R}}
\newcommand{\dom}{\mathop{\mathrm{dom}}}
\newcommand{\epi}{\mathop{\mathrm{epi}}}
\newcommand{\cl}{\mathop{\mathrm{cl}}}

\newcommand{\argmin}{\mathop{\rm argmin}\limits}

\newcommand{\cX}{\mathcal{X}}

\theoremstyle{definition}

\newtheorem{theorem}{Theorem}
\newtheorem{proposition}{Proposition}

\newtheorem{lemma}{Lemma}
\newtheorem{remark}{Remark}
\newtheorem{assumption}{Assumption}

\title{Augmented Lagrangian methods for convex optimization with priority constraints via an infeasibility control framework\thanks{This work was partially supported by the joint project of Kyoto University and Toyota Motor Corporation, titled “Advanced Mathematical Science for Mobility Society.}}
\author{Yuya Yamakawa\thanks{Graduate School of Management, Tokyo Metropolitan University, 1-1 Minami-Osawa, Hachioji, Tokyo 192-0397, Japan, \\ E-mail: yuya@tmu.ac.jp}, Shota Yamanaka\thanks{Toyota Motor Corporation, 1 Toyota-cho, Toyota, Aichi 471-8572, Japan, \\ E-mail:shota\_yamanaka\_aa@mail.toyota.co.jp}, Nobuo Yamashita\thanks{Graduate School of Informatics, Kyoto University, Yoshida-Honmachi, Sakyo-ku, Kyoto 606-8501, Japan, \\ E-mail: nobuo@i.kyoto-u.ac.jp}}
\date{}

\begin{document}
\maketitle

\begin{abstract}
We consider convex optimization problems with prioritized equality constraints, which may be infeasible. In many applications, such as network optimization and image reconstruction, it is often desirable to compute solutions that satisfy higher-priority constraints as much as possible even when no feasible solution exists. To address this issue, we introduce a new solution framework based on the notion of a \emph{hierarchically optimal shift}, which captures the hierarchy among constraints by sequentially minimizing constraint violations according to their priorities. Based on this concept, we define a \emph{hierarchically optimal solution} as an optimal solution of a suitably shifted problem, thereby providing a well-defined notion of optimality even in the absence of feasibility. Furthermore, we propose a novel augmented Lagrangian method equipped with a framework for infeasibility control. The core component is an \emph{infeasibility control problem}, which generates a sequence of approximate shifts converging to the hierarchically optimal shift. This approach enables explicit and systematic handling of prioritized constraint violations, in contrast to existing methods that treat all constraints uniformly. Under suitable assumptions, we show that the generated sequence of shifts converges to the hierarchically optimal shift, and that any accumulation point of the primal iterates is a hierarchically optimal solution. Numerical experiments show that the proposed method achieves solutions consistent with the prescribed constraint hierarchy for both feasible and infeasible cases.
\end{abstract}

\section{Introduction}
We consider the following problem.
\begin{align*}
\begin{aligned} \label{p0_general}
& \mini_{x \in \R^{n}} && f(x)
\\
& \subj && A_1 x - b_1 = 0, ~ A_2 x - b_2 = 0,
\end{aligned} \tag{$\mathrm{P_0}$}
\end{align*}
where $f \colon \R^n \to \R \cup \{ +\infty \}$ is a convex function, and $A_1 \in \R^{m_1 \times n}$, $A_2 \in \R^{m_2 \times n}$, $b_1 \in \R^{m_1}$, and $b_2 \in \R^{m_2}$ are constant matrices and vectors, respectively. It should be noted that problem~\eqref{p0_general} is not necessarily feasible in this study. Moreover, problem~\eqref{p0_general} can represent not only linear equality constraints but also inequality constraints because such constraints can be incorporated into the objective function $f$ via the indicator function of a nonempty closed convex set.
\par
Optimization problems often become infeasible due to the complexity of constraints and the presence of uncertainty. For example, when the number of constraints is large, when mutually conflicting constraints are included, or when observed data contain noise, there may exist no solution that satisfies all the constraints simultaneously. Such infeasibility naturally arises in applications such as optimization problems for transportation in networks~\cite{BDM+2011,YMC2009} and image reconstruction problems~\cite{Co2002,FNW2007,VF2009}. In addition, in model predictive control, even when constraints are infeasible, it is important to compute solutions that satisfy higher-priority constraints as much as possible~\cite{VSF1999,VSJ+2001}.
\par
To address such situations, approaches that compute solutions by minimizing constraint violations have been widely studied. For instance, sequential quadratic programming methods~\cite{BCN2010,BCW2014} and interior-point methods~\cite{DLS2020} have been proposed, where constraint violations are measured using norms or maximum functions and then minimized. However, these methods primarily focus on minimizing constraint violations and do not simultaneously consider the minimization of the objective function.
\par
On the other hand, optimization methods that take the objective function into account have been studied via multiplier methods equipped with the concept of a \emph{shift}~\cite{CG2016, DZ2023}. A shift is a vector added to the constraint functions in order to handle infeasibility. By introducing this concept, the structure of solutions that minimize constraint violations and the convergence properties of algorithms can be analyzed. In particular, Chiche and Gilbert~\cite{CG2016} investigated multiplier methods for infeasible convex quadratic problems and showed that the multiplier-generated shifts converge to an optimal shift, along with results on convergence rates. Dai and Zhang~\cite{DZ2023} extended these results to general convex optimization problems and established not only the convergence of the shifts but also that of the multiplier method itself. However, these studies treat all constraints equally and do not take into account any hierarchy among constraints. In practice, even when no feasible solution exists, it is often desirable to obtain a solution that satisfies more important constraints preferentially.
\par
Such a hierarchical viewpoint has been studied in the context of lexicographic optimization~\cite{Eh2005,Mi1999} and goal programming~\cite{TJR1998}, where priorities are assigned to multiple objective functions. In contrast, the present study introduces priorities into constraint violations and aims to minimize a single objective function while taking such priorities into account.
\par
In this paper, we consider convex optimization problems with prioritized constraints, with particular emphasis on the infeasible case, and propose a new augmented Lagrangian method for such problems. The main contributions of this work are twofold. 
\begin{itemize}
\item First, to define a meaningful solution concept in the infeasible case, we introduce the notion of a \emph{hierarchically optimal shift}, which is defined as a shift that minimizes the violation of the higher-priority constraints and, among such shifts, minimizes the violation of the lower-priority constraints. Based on this shift, we define a \emph{hierarchically optimal solution}, which yields a well-defined notion of optimality even in the absence of feasible points.

\item Second, we propose a novel framework for infeasibility control embedded within the proposed augmented Lagrangian method. Central to this framework is an \emph{infeasibility control problem}, which generates a sequence of approximate shifts converging to the hierarchically optimal shift. This approach enables explicit and systematic handling of prioritized constraint violations, in contrast to existing methods that treat all constraints uniformly. We show that, under suitable assumptions, the generated shifts converge to the hierarchically optimal shift, and that any accumulation point of the primal iterates is a hierarchically optimal solution. Numerical results confirm that the proposed method can find a hierarchically optimal solution that achieves the prescribed hierarchy of constraints in both feasible and infeasible cases.
\end{itemize} 
% In this paper, we consider convex optimization problems with prioritized constraints. In particular, we focus on the case where the problem is infeasible. To define a meaningful solution concept in this setting, we first introduce the notion of a \emph{hierarchically optimal shift}, which is defined as a shift that minimizes the violation of the higher-priority constraints and, among such shifts, minimizes the violation of the lower-priority constraints. By incorporating this shift into the constraint functions, we define a \emph{hierarchically optimal solution} as an optimal solution of the resulting shifted problem, thereby providing a meaningful notion of optimality even when the original problem is infeasible.
% Furthermore, we propose an augmented Lagrangian method for computing such hierarchically optimal solutions. Unlike the standard augmented Lagrangian method, the proposed method introduces an \emph{infeasibility control problem}, which generates a sequence of approximate shifts converging to the hierarchically optimal shift. Under suitable assumptions, we show that the sequence of shifts generated by the proposed method converges to the hierarchically optimal shift, and that any accumulation point of the primal iterates is a hierarchically optimal solution. Numerical experiments demonstrate that the proposed method successfully computes solutions that appropriately reflect the priority structure of the constraints in both feasible and infeasible cases.
\par
The remainder of this paper is organized as follows. In Section~2, we introduce the hierarchically optimal shift and define an approximate hierarchically optimal shift via the infeasibility control problem, and analyze their properties. In Section~3, we propose an augmented Lagrangian method incorporating the infeasibility control problem and establish its global convergence. Section~4 presents numerical experiments to demonstrate the effectiveness of the proposed method.

\subsection*{Notation}
Throughout this paper, we use the following notation. We denote by $\mathbb{N}$ and $\R$ the sets of positive integers and real numbers, respectively. Let $p \in \mathbb{N}$. We denote by $\R^{p}$ the $p$-dimensional Euclidean space. Moreover, we define $\R_{+}^{p} \coloneqq \{ x \in \R^{p} \colon x \geq 0 \}$ and $\R_{++}^{p} \coloneqq \{ x \in \R^{p} \colon x > 0 \}$. 
%For any $v \in \R^{p}$ and $w \in \R^{p}$, the inner product of $v$ and $w$ is defined by $\langle v, w \rangle \coloneqq v^{\top} w$, where $\top$ denotes the transpose of vectors or matrices. 
We denote by $\Vert v \Vert$ the Euclidean norm of $v \in \R^p$, that is, $\Vert v \Vert \coloneqq \sqrt{v^{\top} v}$, where $\top$ denotes the transpose of vectors or matrices. The symbol $e$ denotes the vector of all ones of appropriate dimension. Let $\varphi \colon \R^{p} \to \R$ be a convex function on $\R^p$. The effective domain of $\varphi$ is represented as $\dom \varphi$. The subdifferential of $\varphi$ at $x \in \R^p$ is written by $\partial \varphi(x)$. The conjugate function of $\varphi$ is defined by $\varphi^{\ast}(y) \coloneqq \sup \{ y^{\top} x - \varphi (x) \colon x \in \dom \varphi \}$. The epigraph of $\varphi$ is denoted by $\epi \varphi$. The closure of $\varphi$ is represented as $\cl \varphi$. For a set $X \subset \R^{p}$, the relative interior of $X$ is denoted by ${\rm ri}(X)$, and the closure of $X$ is written as $\cl (X)$. For a non-empty convex set $C \subset \R^{p}$, the normal cone of $C$ at $x$ is represented by ${\cal N}_{C}(x)$. For a non-empty set $X \subset \R^{p}$, the indicator function on $X$ is written as $\delta_{X}$. For $\varepsilon \geq 0$, a set $Z \subset \R^p$, and a function $\psi \colon Z \to \R$, we say that $\bar{z}$ is an $\varepsilon$-optimal solution of the problem of minimizing $\psi$ over $Z$ if $\bar{z} \in Z$ and $\psi(\bar{z}) \leq \inf \{ \psi(z) \colon z \in Z \} + \varepsilon$.

\section{Shift for convex optimization problems with priority constraints}

In this section, we consider the case where problem~\eqref{p0_general} is infeasible. 
First, for the infeasible problem~\eqref{p0_general}, we define a hierarchically optimal shift, which represents the best feasible shift when the priority of constraints is taken into account. 
Second, in order to find an optimal solution to the problem equipped with a hierarchically optimal shift, we introduce an approximate hierarchically optimal shift that converges to the hierarchically optimal one, and describe its theoretical properties.

\subsection{Shifted problem and hierarchically optimal shift}
\noindent
For $s_1 \in \R^{m_1}$ and $s_2 \in \R^{m_2}$, we introduce the following problem related to~\eqref{p0_general}:
\begin{align}
\label{p1_linear_slack_org}
\begin{aligned}
& \mini_{x \in \R^{n}} && f(x) 
\\
& \subj && A_1 x - b_1 + s_1 = 0, ~ A_2 x - b_2 + s_2 = 0, ~ x \in \dom f.
\end{aligned}
\tag{$\mathrm{P}_s$}
\end{align}
Here, $(s_1, s_2)$ is called a \emph{shift}, and problem~\eqref{p1_linear_slack_org} is referred to as the \emph{shifted problem}. Note that the shift is treated as a given parameter and is fixed when solving problem~\eqref{p1_linear_slack_org}. Note also that there always exists a shift that makes problem~\eqref{p1_linear_slack_org} feasible. 
Indeed, since $f$ is proper, we can take $\bar{x} \in \dom f$ and define $s_1 \coloneqq b_1 - A_1 \bar{x}$ and $s_2 \coloneqq b_2 - A_2 \bar{x}$, which ensures that $\bar{x}$ is feasible for \eqref{p1_linear_slack_org}. 
Such a shift is called a \emph{feasible shift}, and the set of all feasible shifts is referred to as the \emph{feasible shift set}, defined by
\begin{align*}
{\cal S} \coloneqq \left\{ (s_{1}, s_{2}) \in \R^{m_{1}} \times \R^{m_{2}} \colon \, \exists x \in \dom f, \, s_{1} = b_{1} - A_{1} x, \, s_{2} = b_{2} - A_{2} x \right\}.
\end{align*}

In this study, we propose an augmented Lagrangian method for problem~\eqref{p0_general}, where the feasibility is not assumed. However, since problem~\eqref{p0_general} may be infeasible, one may have to deal with a problem that has no optimal solution. To overcome this difficulty, we instead consider applying the augmented Lagrangian method to the shifted problem~(P$_{s^{(k)}}$) with a sequence $\{ s^{(k)} \}$ of shifts. In this setting, the proposed method is required not only to solve subproblems of the augmented Lagrangian method for~(P$_{s^{(k)}}$) but also to simultaneously determine a shift $s^{(k)}$ that properly reflects the priority structure between the higher- and lower-priority constraints. Therefore, in designing the proposed method, we need to consider a situation in which the shift varies at each iteration, and thus the following assumption is required to guarantee the solvability of problem~\eqref{p1_linear_slack_org} for every feasible shift.

\begin{assumption} \label{ass:ps_solvable}
Problem~\eqref{p1_linear_slack_org} has at least one optimal solution for each $s \in {\cal S}$.
\end{assumption}

\par
Taking into account the priority of the constraints, we define a \emph{hierarchically optimal shift} $s^{\ast} \coloneqq (s_1^{\ast}, s_2^{\ast})$ as follows:
\begin{align}
& s_1^{\ast} \coloneqq b_1 - A_1 x^{\dag}, \quad x^{\dag} \in \argmin_{x \in \dom f} \frac{1}{2} \Vert b_1 - A_1 x \Vert^2, \label{eq:true_solution} 
\\
& s_2^{\ast} \coloneqq b_2 - A_2 x^{\ddag}, \quad x^{\ddag} \in \argmin_{x \in \cX_1} \frac{1}{2} \Vert b_2 - A_2 x\Vert^2. \label{eq:true_solution2}
\end{align}
where ${\cal X}_1 \coloneqq \{x \in \dom f \colon s_1^{\ast} = b_1 - A_1 x \}$.
%Note that the hierarchically optimal shift $s^{\ast}$ is a feasible shift that minimizes constraint violations in a hierarchical manner. 
To ensure that $s_1^{\ast}$ and $s_2^{\ast}$ are well-defined, we require the existence of $x^{\dag}$ and $x^{\ddag}$. Accordingly, we impose the following assumption:
\begin{assumption} \label{twoprob_solvable}
The following conditions hold:
\begin{align*}
\argmin_{x \in \dom f} \frac{1}{2} \Vert b_1 - A_1 x \Vert^2 \not = \emptyset, \quad \argmin_{x \in \cX_1} \frac{1}{2} \Vert b_2 - A_2 x \Vert^2 \not = \emptyset.
\end{align*}
\end{assumption}

\begin{remark}
The first condition of Assumption~\ref{twoprob_solvable} ensures that $s_1^{\ast}$ is well-defined. Then, by using $s_1^{\ast}$, the set $\cX_1$ is well-defined, and the second condition ensures that $s_2^{\ast}$ is also well-defined.
\end{remark}

\begin{remark}
It should be emphasized that the closedness of $\dom f$ is not assumed in the proofs of Lemma~\ref{lem:true_sol_unique}, Lemma~\ref{lem:equivalent}, Proposition~\ref{pro:solvability_Rsigma}, and Lemma~\ref{lem:caA_calB}. The analysis of this study is based instead on Assumption~\ref{twoprob_solvable}, which directly postulates the existence of the hierarchically optimal shift. Accordingly, we include explicit proofs for these results, even though some of them would be immediate if $\dom f$ were closed. We also note that Assumption~\ref{twoprob_solvable} can be satisfied even when $\dom f$ is not closed, so that the proposed framework applies to such settings as well.
\end{remark}

\noindent
The shift $s_1^{\ast}$ defined above minimizes the violation of the higher-priority constraints, while $s_2^{\ast}$ minimizes the violation of the lower-priority constraints over points $x \in \cX_1$ that already achieve the minimal violation of the higher-priority constraints. 
It can be shown, as stated in the following lemma, that $s_1^{\ast}$ and $s_2^{\ast}$ together constitute the unique optimal solution to a certain optimization problem.

\begin{lemma}\label{lem:true_sol_unique}
Suppose that Assumption~{\rm \ref{twoprob_solvable}} is satisfied. Then, the following two equalities hold:
\begin{align*}
& \min \left \{ \frac{1}{2} \Vert s_{1} \Vert^{2} \colon s_{1} \in {\cal S}_{1} \right \} = \min \left\{ \frac{1}{2} \Vert A_{1} x - b_{1} \Vert^{2} \colon x \in \dom f \right\},
\\
& \min \left \{ \frac{1}{2} \Vert s_{2} \Vert^{2} \colon s_{2} \in {\cal S}_{2} \right \} = \min \left\{ \frac{1}{2} \Vert A_{2} x - b_{2} \Vert^{2} \colon x \in {\cal X}_{1} \right\},
\end{align*}
where ${\cal S}_{1} \coloneqq \{ s_1 \in \R^{m_{1}} \colon \exists x \in \dom f, \, s_1 = b_{1} - A_{1} x \}$ and ${\cal S}_{2} \coloneqq \{ s_2 \in \R^{m_{2}} \colon \exists x \in {\cal X}_{1}, \, s_2 = b_{2} - A_{2} x \}$. Moreover, the shifts $s_{1}^{\ast}$ and $s_{2}^{\ast}$ defined by~\eqref{eq:true_solution} and~\eqref{eq:true_solution2}, respectively, are the unique solutions of $\min \{ \frac{1}{2} \Vert s_{1} \Vert^{2} \colon s_{1} \in {\cal S}_{1} \}$ and $\min \{ \frac{1}{2} \Vert s_{2} \Vert^{2} \colon s_{2} \in {\cal S}_{2} \}$, respectively. 
\end{lemma}
\begin{proof}
To begin with, let $\alpha_{1} \coloneqq \inf \{ \frac{1}{2}\Vert s_{1} \Vert^{2} \colon s_{1} \in {\cal S}_{1} \}$ and $\beta_{1} \coloneqq \min \{ \frac{1}{2} \Vert b_{1} - A_{1} x \Vert^{2} \colon x \in \dom f \}$. Moreover, let $s_{1}^{\ast} \in \R^{m_{1}}$ and $x^{\dag} \in \R^{n}$ be defined by~\eqref{eq:true_solution}. In what follows, we will show that $\alpha_{1} = \beta_{1}$. Because $x^{\dag} \in \dom f$ is satisfied, we readily have
\begin{align}
b_{1} - A_{1} x^{\dag} \in {\cal S}_{1} = \{ b_{1} - A_{1} x \in \R^{m_{1}} \colon x \in \dom f \} \label{b1A1xdag_in_S1}
\end{align}
It then follows from the definitions of $\alpha_{1}$, $\beta_{1}$, and $x^{\dag}$ that
\begin{align}
\alpha_{1} \leq \frac{1}{2} \Vert b_{1} - A_{1} x^{\dag} \Vert^{2} = \beta_{1}. \label{eq:alpha1_beta1}
\end{align}
Let $\varepsilon > 0$ be arbitrary. The definition of $\alpha_{1}$ guarantees the existence of $s_{\varepsilon} \in {\cal S}_{1}$ satisfying $\frac{1}{2} \Vert s_{\varepsilon} \Vert^{2} < \alpha_{1} + \varepsilon$. By noting $s_{\varepsilon} \in {\cal S}_{1}$, there exists $x_{\varepsilon} \in \dom f$ such that $s_{\varepsilon} = b_{1} - A_{1} x_{\varepsilon}$. From $x_{\varepsilon} \in \dom f$ and $x^{\dag} \in \argmin \{ \frac{1}{2} \Vert b_1 - A_1 x \Vert^2 \colon x \in \dom f \}$, we obtain
\begin{align*}
\beta_{1} = \frac{1}{2} \Vert b_{1} - A_{1} x^{\dag} \Vert^{2} \leq \frac{1}{2} \Vert b_{1} - A_{1} x_{\varepsilon} \Vert^{2} = \frac{1}{2} \Vert s_{\varepsilon} \Vert^{2} < \alpha_{1} + \varepsilon, 
\end{align*}
that is, $\beta_{1} \leq \alpha_{1}$ because $\varepsilon$ is an arbitrary number. Hence, we have from~\eqref{eq:alpha1_beta1} that $\alpha_{1} = \beta_{1}$.
\par
Next, we will show that $s_1^{\ast}$ is a unique optimal solution of $\min \{ \frac{1}{2} \Vert s_{1} \Vert^{2} \colon s_{1} \in {\cal S}_{1} \}$. Now, we notice that $\inf\{ \frac{1}{2} \Vert s_{1} \Vert^{2} \colon s_{1} \in {\cal S}_{1} \} = \alpha_{1} = \beta_{1} = \frac{1}{2} \Vert b_{1} - A_{1} x^{\dag} \Vert^{2} = \frac{1}{2} \Vert s_{1}^{\ast} \Vert^{2}$, and recall that $s_{1}^{\ast} = b_{1} - A_{1} x^{\dag} \in {\cal S}_{1}$ from~\eqref{b1A1xdag_in_S1}. These facts mean that $s_{1}^{\ast}$ solves $\min \{ \frac{1}{2} \Vert s_{1} \Vert^{2} \colon s_{1} \in {\cal S}_{1} \}$. Moreover, the uniqueness of $s_{1}^{\ast}$ can be easily verified because $\min \{ \frac{1}{2} \Vert s_{1} \Vert^{2} \colon s_{1} \in {\cal S}_{1} \}$ is a convex optimization problem with a strongly convex objective function. The same argument applies to the proof of the equation $\min \{ \frac{1}{2}\Vert s_{2} \Vert^{2} \colon s_{2} \in {\cal S}_{2} \} = \min \{ \frac{1}{2} \Vert b_{2} - A_{2} x \Vert^{2} \colon x \in {\cal X}_{1} \}$ and the uniqueness of $s_{2}^{\ast}$.
\end{proof}

\par
Now, we consider the following optimization problem associated with the hierarchically optimal shift $s^{\ast} = (s_1^{\ast}, s_2^{\ast})$:
\begin{align} 
\begin{aligned}
\label{p1_linear_slack_opt}
&\mini_{x \in \R^{n}} && f(x) 
\\
& \subj && A_1 x - b_1 + s_1^{\ast} = 0, ~ A_2 x - b_2 + s_2^{\ast} = 0, ~ x \in \dom f.
\end{aligned}
\tag{$\mathrm{P}_{s^{\ast}}$}
\end{align}

\noindent
In this study, we impose the following assumption on problem~\eqref{p1_linear_slack_org}.
\begin{assumption}\label{ass:sast_solvable}
~%The following two conditions hold:
\begin{itemize}
\item[{\rm (i)}] The function $f$ is closed proper convex.
\item[{\rm (ii)}] Problem~\eqref{p1_linear_slack_opt} satisfies Slater's constraint qualification, i.e., there exists $\bar{x} \in {\rm ri}(\dom f)$ such that $A_1 \bar{x} - b_1 + s_1^{\ast} = 0$ and $A_2 \bar{x} - b_2 + s_2^{\ast} = 0$.
\end{itemize}
\end{assumption}

\noindent
The objective of this study is to find an optimal solution of problem~\eqref{p1_linear_slack_opt}, for which we propose an augmented Lagrangian method. Note that the solvability of problem~\eqref{p1_linear_slack_opt} is guaranteed by Assumption~\ref{ass:ps_solvable}. 
\par
We define any feasible solution of problem~\eqref{p1_linear_slack_opt} as a \emph{hierarchically minimal violation point}, and its optimal solution as a \emph{hierarchically optimal solution}. By definition, a hierarchically minimal violation point is a point that minimizes the violation of the lower-priority constraints among those that minimize the violation of the priority constraints. 
Moreover, a hierarchically optimal solution is a point that minimizes the objective function among all hierarchically minimal violation points.

\subsection{Approximate hierarchically optimal shifts via an infeasibility control problem} \label{sec:app_opt_shift}

The proposed method is required not only to compute an optimal solution of problem~\eqref{p1_linear_slack_opt}, but also to generate a sequence of shifts that converges to the hierarchically optimal shift $s^{\ast}$. To this end, we first introduce an infeasibility control problem. Second, by using the problem, we define an \emph{approximate hierarchically optimal shift} and investigate its fundamental properties.

To begin with, we consider the following \emph{infeasibility control problem}:
\begin{align}
\label{p_sigma}
\begin{aligned}
& \mini_{(x,s_1,s_2) \in {\cal W}} && \frac{\sigma_1}{2} \Vert s_1 \Vert^2 + \frac{\sigma_2}{2} \Vert s_2 \Vert^2
\\
& \subj && s_1 = b_1 - A_1 x, ~ s_2 = b_2 - A_2 x, ~ x \in \dom f,
\end{aligned}
\tag{$\mathrm{R}_\sigma$}
\end{align}
where ${\cal W} \coloneqq \R^{n} \times \R^{m_1} \times \R^{m_2}$, and $\sigma \coloneqq (\sigma_1, \sigma_2)$ is a parameter satisfying $\sigma_1 > 0$ and $\sigma_2 > 0$.

\noindent
In what follows, we discuss the solvability of~\eqref{p_sigma}. For simplicity, we define $m \coloneqq m_1 + m_2$. Moreover, we define $M_{\sigma} \in \R^{m \times m}$, $A \in \R^{m \times n}$, and $b \in \R^{m}$ as
\begin{align*}
M_{\sigma} \coloneqq
\begin{bmatrix}
\sigma_{1} I & O
\\
O & \sigma_{2} I
\end{bmatrix}
, \quad A \coloneqq
\begin{bmatrix}
A_{1}
\\
A_{2}
\end{bmatrix}
, \quad b \coloneqq
\begin{bmatrix}
b_{1}
\\
b_{2}
\end{bmatrix}
,
\end{align*}
\noindent
respectively. To ensure the solvability~\eqref{p_sigma}, we need the following assumption.
\begin{assumption} \label{ass:Adomf_closed}
The set $A(\dom f) = \{ A x \in \R^{m} \colon x \in \dom f \}$ is closed.
\end{assumption}

\begin{remark}
Examples of sufficient conditions for Assumption~\ref{ass:Adomf_closed} to hold are that $\dom f = \R^n$ and that $\dom f$ is a compact set of $\R^{n}$.
\end{remark}

\noindent
Moreover, we show that~\eqref{p_sigma} is equivalent to the following problem:
\begin{align}
\label{opt:Rsigma2}
\begin{aligned}
& \mini_{s \in \R^{m}} && \frac{1}{2} s^{\top} M_{\sigma} s
\\
& \subj && s \in {\cal S},
\end{aligned}
\tag{${\cal R}_{\sigma}$}
\end{align}
where we recall that ${\cal S} = \{ b - A x \in \R^{m} \colon x \in \dom f \}$.

\begin{lemma} \label{lem:equivalent}
Let $\sigma \in \R^2_{++}$. The following two statements hold:
\begin{itemize}
\item[{\rm (a)}] If $\bar{s} \coloneqq (\bar{s}_{1}, \bar{s}_{2}) \in \R^{m_1} \times \R^{m_2}$ is an optimal solution of~\eqref{opt:Rsigma2}, then there exists $\bar{x} \in \R^{n}$ such that $(\bar{x}, \bar{s}_{1}, \bar{s}_{2})$ solves~\eqref{p_sigma}.
\item[{\rm (b)}] If $(x^{\ast}, s_{1}^{\ast}, s_{2}^{\ast}) \in \R^{n} \times \R^{m_1} \times \R^{m_2}$ is an optimal solution of~\eqref{p_sigma}, then $s^{\ast} \coloneqq (s_{1}^{\ast}, s_{2}^{\ast})$ solves~\eqref{opt:Rsigma2}. 
\end{itemize}
\end{lemma}
\begin{proof}
Let $\sigma \in \R^2_{++}$ be arbitrary. First, we prove item~(a). Let us define $\alpha \coloneqq \inf \{ \phi(x,s) \colon (x, s) \in {\cal V} \}$, where $\phi(x, s) \coloneqq \frac{1}{2} s^{\top} M_{\sigma} s$ and ${\cal V} \coloneqq \{ (x, s) \in \R^{n} \times \R^{m} \colon x \in \dom f, \, s=b-Ax \}$. Since $\bar{s} \in {\cal S}$ holds, there exists $\bar{x} \in \R^{n}$ such that
\begin{align} \label{eq:bars_barx}
\bar{s} = b - A \bar{x}, \quad \bar{x} \in \dom f.   
\end{align}
Let $\varepsilon > 0$ be arbitrary. From the definition of $\alpha$, there exists $(x_{\varepsilon}, s_{\varepsilon}) \in \R^{n} \times \R^{m}$ such that
\begin{align} \label{def:gamma}
\phi(x_{\varepsilon}, s_{\varepsilon}) < \alpha + \varepsilon, \quad s_{\varepsilon} = b - A x_{\varepsilon}, \quad x_{\varepsilon} \in \dom f.
\end{align}
Noting~\eqref{def:gamma} derives $s_{\varepsilon} \in {\cal S}$, namely, $s_{\varepsilon}$ is a feasible point of~\eqref{opt:Rsigma2}. Thus, we have 
\begin{align} \label{ineq:psi_barxbars}
\phi(\bar{x}, \bar{s}) = \frac{1}{2} \bar{s}^{\top} M_{\sigma} \bar{s} \leq \frac{1}{2} s_{\varepsilon}^{\top} M_{\sigma} s_{\varepsilon} = \phi(x_{\varepsilon}, s_{\varepsilon}).
\end{align}
Combining~\eqref{def:gamma} and~\eqref{ineq:psi_barxbars} yields $\phi(\bar{x}, \bar{s}) < \alpha + \varepsilon$, that is, $\phi(\bar{x}, \bar{s}) \leq \alpha$.
Meanwhile, since $(\bar{x}, \bar{s}) \in {\cal V}$ holds from~\eqref{eq:bars_barx}, we readily have $\alpha \leq \phi(\bar{x}, \bar{s})$. These facts and the definition of $\alpha$ mean that $(\bar{x}, \bar{s})$ is an optimizer of $\min \{ \phi(x, s) \colon (x,s) \in {\cal V} \}$. Since~\eqref{p_sigma} is equal to $\min \{ \phi(x, s) \colon (x,s) \in {\cal V} \}$, item~(a) is proven. 
\par
Next, we show item~(b). Let $\beta \coloneqq \inf \{ \psi(s) \colon s \in {\cal S} \}$, where $\psi(s) \coloneqq \frac{1}{2} s^{\top} M_{\sigma} s$. We take an arbitrary number $\delta > 0$. The definition of $\beta$ ensures the existence of $s_{\delta} \in \R^{m}$ satisfying
\begin{align} \label{ineq:s_delta}
\psi(s_{\delta}) < \beta + \delta, \quad s_{\delta} \in {\cal S}.
\end{align}
Moreover, from $s_{\delta} \in {\cal S}$, there exists $x_{\delta} \in \R^{n}$ such that $x_{\delta} \in \dom f$ and $s_{\delta} = b - A x_{\delta}$, that is, $(x_{\delta}, s_{\delta}) \in {\cal V}$. Since $(x^{\ast}, s_1^{\ast}, s_2^{\ast})$ is an optimal solution of~\eqref{p_sigma}, this fact and~\eqref{ineq:s_delta} imply
\begin{align} \label{ineq:psi_beta}
\psi(s^{\ast}) \leq \psi(s_{\delta}) < \beta + \delta, \quad {\rm i.e.,} \quad \psi(s^{\ast}) \leq \beta
\end{align}
Meanwhile, note that $s^{\ast} \in {\cal S}$ because $x^{\ast} \in \dom f$ holds, namely, $s^{\ast}$ is a feasible point of~\eqref{opt:Rsigma2}. Hence, we obtain $\beta \leq \psi(s^{\ast})$. It then follows from~\eqref{ineq:psi_beta} and the definition of $\beta$ that $\inf \{ \psi(s) \colon s \in {\cal S} \} = \psi(s^{\ast})$. These results derive that $s^{\ast} \coloneqq (s_{1}^{\ast}, s_{2}^{\ast})$ solves~\eqref{opt:Rsigma2}.
\end{proof}

\begin{proposition} \label{pro:solvability_Rsigma}
If Assumption~{\rm \ref{ass:Adomf_closed}} holds, problem~\eqref{opt:Rsigma2} has a unique optimal solution for any $\sigma \in \R_{++}^2$.
\end{proposition}
\begin{proof}
Let $\sigma \in \R^2_{++}$ be arbitrary. From item~(a) of Lemma~\ref{lem:equivalent}, it is sufficient to show that problem~\eqref{opt:Rsigma2} has at least one optimizer. From the definition of $\inf \{ \frac{1}{2} s^{\top} M_{\sigma} s \colon s \in {\cal S} \}$, there exists $\{ s^{(j)} \} \subset {\cal S}$ such that
\begin{align}
\inf \left\{ \frac{1}{2} s^{\top} M_{\sigma} s \colon s \in {\cal S} \right\} = \lim_{j \to \infty} \frac{1}{2} (s^{(j)})^{\top} M_{\sigma} s^{(j)}. \label{lim:sj_gamma}
\end{align}
Since $\{ s^{(j)} \} \subset {\cal S}$ holds, %there exists $\{ x^{(j)} \} \subset \dom f$ such that 
\begin{align} \label{eq:sj_bAxj}
\exists \{ x^{(j)} \} \subset \dom f, \quad s^{(j)} = b - A x^{(j)} \quad  \forall j \in \mathbb{N}.   
\end{align}
Moreover, let $\{ p^{(j)} \} \subset \R^{n}$ and $\{ q^{(j)} \} \subset \R^{n}$ be sequences satisfying
\begin{align} \label{xj_decomposition}
x^{(j)} = p^{(j)} + q^{(j)},\quad p^{(j)} \in {\rm Ker}(A)^{\perp}, \quad q^{(j)} \in {\rm Ker}(A) \quad \forall j \in \mathbb{N}.   
\end{align}
Hence, we have
\begin{align} \label{sj_decomposition}
A p^{(j)} = A(p^{(j)} + q^{(j)}) = A x^{(j)} = b - s^{(j)} \quad \forall j \in \mathbb{N}.
\end{align}
Now, let us define $\varphi \colon {\rm Ker}(A)^{\perp} \to \R^{m}$ as $\varphi(v) \coloneqq A v$. We can easily see that $\varphi$ is injective on ${\rm Ker}(A)^{\perp}$, that is, there exists $c > 0$ such that $c \Vert v \Vert \leq \Vert \varphi(v) \Vert$ for all $v \in {\rm Ker}(A)^{\perp}$. It then follows from $\{ p^{(j)} \} \subset {\rm Ker}(A)^{\perp}$ and~\eqref{sj_decomposition} that
\begin{align} \label{phi:injection}
\Vert p^{(j)} \Vert \leq \frac{1}{c} \Vert \varphi(p^{(j)}) \Vert = \frac{1}{c} \Vert A p^{(j)} \Vert \leq \frac{1}{c} \left( \Vert b \Vert + \Vert s^{(j)} \Vert \right) \quad \forall j \in \mathbb{N}.
\end{align}
Since $\{ s^{(j)} \}$ is bounded from~\eqref{lim:sj_gamma} and the positive definiteness of $M_{\sigma}$, the boundedness of $\{ p^{(j)} \}$ is also ensured by~\eqref{phi:injection}. From this fact, without loss of generality, we can assume that $p^{(j)} \to \bar{p}$ as $j \to \infty$, where $\bar{p} \in \R^{n}$ is some vector. Using~\eqref{xj_decomposition} yields 
\begin{align} \label{lim:vjvbar}
v^{(j)} \coloneqq Ax^{(j)} = Ap^{(j)} \longrightarrow A\bar{p} \eqqcolon \bar{v} \quad (j \to \infty).    
\end{align}
Now, recall that $\{ x^{(j)} \} \subset \dom f$ is verified by~\eqref{eq:sj_bAxj}. This fact and~\eqref{lim:vjvbar} imply that $\{ v^{(j)} \} \subset A(\dom f)$ converges to $\bar{v}$. It then follows from the closedness of $A(\dom f)$ that $\bar{v} \in A(\dom f)$, that is, there exists $\bar{x} \in \dom f$ such that $\bar{v} = A \bar{x}$. Now, we define $\bar{s} \coloneqq b - \bar{v}$. Notice that $\bar{s} \in {\cal S}$ is satisfied because $\bar{x} \in \dom f$ and $\bar{s} = b - A \bar{x}$. Moreover, combining~\eqref{eq:sj_bAxj} and~\eqref{lim:vjvbar} derives $s^{(j)} = b - A x^{(j)} = b - v^{(j)} \to b - \bar{v} = \bar{s}$ as $j \to \infty$. These results and~\eqref{lim:sj_gamma} lead to
\begin{align*}
\inf \left\{ \frac{1}{2} s^{\top} M_{\sigma} s \colon s \in {\cal S} \right\} = \lim_{j \to \infty} \frac{1}{2} (s^{(j)})^{\top} M_{\sigma} s^{(j)} = \frac{1}{2} \bar{s}^{\top} M_{\sigma} \bar{s}, \quad \bar{s} \in {\cal S}.
\end{align*}
Therefore, we conclude that $\bar{s}$ solves~\eqref{opt:Rsigma2}. The uniqueness of $\bar{s}$ is derived from the strong convexity of the objective function.
\end{proof}

\noindent
Combining Lemma~\ref{lem:equivalent} and Proposition~\ref{pro:solvability_Rsigma} ensures that problem~\eqref{p_sigma} has at least one optimal solution. For any $\sigma = (\sigma_1, \sigma_2) \in \R_{++}^2$, let $(\bar{x}(\sigma), \bar{s}_1(\sigma), \bar{s}_2(\sigma))$ be an optimal solution of problem~\eqref{p_sigma}. For simplicity, we define $\bar{s}(\sigma) \coloneqq (\bar{s}_1(\sigma), \bar{s}_2(\sigma))$. In the following, we call $\bar{s}(\sigma)$ an approximate hierarchically optimal shift. 
% \par
% Since Proposition~\ref{pro:solvability_Rsigma} ensures that problem~\eqref{opt:Rsigma2} has an unique optimal solution, it is denoted by $\bar{s}(\sigma) \coloneqq (\bar{s}_1(\sigma), \bar{s}_2(\sigma)) \in \R^{m_1} \times \R^{m_2}$. Then, by item~(a) of Lemma~\ref{lem:equivalent}, there exists $\bar{x}(\sigma) \in \R^{n}$ such that $(\bar{x}(\sigma), \bar{s}_1(\sigma), \bar{s}_2(\sigma))$ be an optimal solution of problem~\eqref{p_sigma}. Thus, the solvability of~\eqref{p_sigma} is verified. In the following, we call $\bar{s}(\sigma)$ an approximate hierarchically optimal shift. 

\begin{remark}
In the proposed method, at each iteration $k \in \mathbb{N} \cup \{0\}$, we consider computing an approximate hierarchically optimal shift $\bar{s}(\sigma^{(k)}) = (\bar{s}_1(\sigma^{(k)}), \bar{s}_2(\sigma^{(k)}))$ using suitably chosen parameters $\sigma^{(k)} = (\sigma_1^{(k)}, \sigma_2^{(k)})$. Since the approximate hierarchically optimal shift $\bar{s}(\sigma^{(k)})$ is obtained from the infeasibility control problem, it is expected that, by appropriately selecting the sequences $\{ \sigma_1^{(k)} \}$ and $\{ \sigma_2^{(k)} \}$, the sequence $\{ \bar{s}(\sigma^{(k)}) \}$ converges to $s^{\ast}$. This framework enables automatic handling of prioritized constraint violations, and it is referred to as the \emph{infeasibility control framework}. In what follows, we discuss the convergence properties of the sequence $\{ \bar{s}(\sigma^{(k)}) \}$ under the following assumption.    
\end{remark}

\begin{remark}
In principle, one may solve the shifted problem~\eqref{p1_linear_slack_opt} directly if the hierarchically optimal shift $s^{\ast}$ were known in advance. In some cases, the shift $s^{\ast}$ can be obtained by solving problems~\eqref{eq:true_solution} and~\eqref{eq:true_solution2}. In contrast, the infeasibility control framework does not rely on such an exact preprocessing step and instead generates, simultaneously with solving the problem, a sequence of approximate shifts $\{ \bar{s}(\sigma^{(k)})\}$ that converges to $s^{\ast}$.   
\end{remark}

\begin{assumption}\label{asm:sigma_frac_zero}
The sequences $\{ \sigma_{1}^{(k)} \}$ and $\{ \sigma_{2}^{(k)} \}$ satisfy 
\begin{align*}
\lim_{k \to \infty} \frac{\sigma_1^{(k)}}{\sigma_2^{(k)}} = \infty, \qquad \frac{\sigma_{1}^{(k-1)}}{\sigma_{2}^{(k-1)}} \leq \frac{\sigma_{1}^{(k)}}{\sigma_{2}^{(k)}} \quad \forall k \in \mathbb{N}.
\end{align*}
\end{assumption}

\noindent
For simplicity, we denote $\eta^{(k)} \coloneqq \sigma_1^{(k)}/\sigma_2^{(k)}$ for each $k \in \mathbb{N} \cup \{ 0 \}$, and let $g \colon \R^{m_1} \times \R^{m_2} \to \R$ and $h \colon \R^{m_1} \times \R^{m_2} \to \R$ be defined as $g(s_1, s_2) \coloneqq \frac{1}{2} \Vert s_2 \Vert^{2}$ and $h(s_1, s_2) \coloneqq \frac{1}{2} \Vert s_1 \Vert^{2}$, respectively. Moreover, we define the sets ${\cal A}$ and ${\cal B}$ as follows:
\begin{align*}
{\cal A} \coloneqq \argmin \left\{ g(s_1, s_2) \colon (s_1, s_2) \in {\cal B} \right\}, \quad {\cal B} \coloneqq \argmin \left\{ h(s_1, s_2) \colon (s_1, s_2) \in {\cal S} \right\}.
\end{align*}

\noindent
From now on, we show the convergence of the approximate hierarchically optimal shift $\bar{s}(\sigma^{(k)})$ under Assumption~\ref{asm:sigma_frac_zero}.

\begin{lemma} \label{lem:caA_calB}
Suppose that Assumption~{\rm \ref{twoprob_solvable}} holds. Then, the sets ${\cal A}$ and ${\cal B}$ satisfy ${\cal A} = \{ (s_{1}^{\ast}, s_{2}^{\ast}) \}$ and $(s_{1}^{\ast}, s_{2}^{\ast}) \in {\cal B}$, respectively.
\end{lemma}
\begin{proof}
To begin with, Lemma~\ref{lem:true_sol_unique} ensures that
\begin{align}\label{eq:s1ast_s2ast}
s_{1}^{\ast} = \argmin \left\{ \frac{1}{2} \Vert s_{1} \Vert^{2} \colon s_{1} \in {\cal S}_{1} \right\}, \quad s_{2}^{\ast} = \argmin \left\{ \frac{1}{2} \Vert s_{2} \Vert^{2} \colon s_{2} \in {\cal S}_{2} \right\}, 
\end{align}
where we recall that ${\cal S}_{1} = \{ s_1 \in \R^{m_1} \colon \exists x \in \dom f, \, s_1 = b_1 - A_1 x \}$ and ${\cal S}_{2} = \{ s_2 \in \R^{m_2} \colon \exists x \in \dom f, \, s_{1}^{\ast} = b_1 - A_1 x, \, s_2 = b_2 - A_2 x \}$.
\par
First, we show that ${\cal A} \subset \{ (s_1^{\ast}, s_2^{\ast}) \}$ and $(s_{1}^{\ast}, s_{2}^{\ast}) \in {\cal B}$. Let $(\bar{s}_{1}, \bar{s}_{2}) \in {\cal A}$ be arbitrary. From the definitions of ${\cal A}$ and ${\cal B}$, it is clear that 
\begin{align} \label{s1s2inBS}
(\bar{s}_{1}, \bar{s}_{2}) \in {\cal B} \cap {\cal S}.    
\end{align}
Now, it follows from $s_{2}^{\ast} \in {\cal S}_{2}$ that 
\begin{align} \label{sast_in_S}
\exists x^{\ast} \in \dom f, \quad s_{1}^{\ast} = b_1 - A_1 x^{\ast}, \quad s_{2}^{\ast} = b_2 - A_2 x^{\ast}, \quad {\rm i.e.,} \quad (s_{1}^{\ast}, s_{2}^{\ast}) \in {\cal S}.
\end{align}
Since $(\bar{s}_{1}, \bar{s}_{2}) \in {\cal B}$ holds from~\eqref{s1s2inBS}, we obtain
\begin{align} \label{s1barleqs1ast}
\frac{1}{2} \Vert \bar{s}_{1} \Vert^2 = h(\bar{s}_1, \bar{s}_2) \leq h(s_1^{\ast}, s_2^{\ast}) = \frac{1}{2} \Vert s_{1}^{\ast} \Vert^2.
\end{align}
Meanwhile, because~\eqref{s1s2inBS} guarantees $(\bar{s}_{1}, \bar{s}_{2}) \in {\cal S}$,
\begin{align} \label{s1s2_in_calS}
\exists \bar{x} \in \dom f, \quad \bar{s}_1 = b_1 - A_1 \bar{x}, \quad \bar{s}_2 = b_2 - A_2 \bar{x},
\end{align}
which implies that $\bar{s}_{1} \in {\cal S}_{1}$. Thus, exploiting the first equality of~\eqref{eq:s1ast_s2ast} yields $\frac{1}{2} \Vert s_{1}^{\ast} \Vert^{2} \leq \frac{1}{2} \Vert \bar{s}_{1} \Vert^{2}$. It then follows from~\eqref{s1barleqs1ast} that $\min \{ \frac{1}{2} \Vert s_{1} \Vert^2 \colon s_{1} \in {\cal S}_{1} \} = \frac{1}{2} \Vert s_{1}^{\ast} \Vert^{2} = \frac{1}{2} \Vert \bar{s}_{1} \Vert^{2}$. Then, again using~\eqref{eq:s1ast_s2ast} yields 
\begin{align} \label{eq:s1bar_s1ast}
\bar{s}_{1} \in \argmin \left\{ \frac{1}{2} \Vert s_{1} \Vert^{2} \colon s_{1} \in {\cal S}_{1} \right\}, \quad {\rm i.e.,} \quad \bar{s}_{1} = s_{1}^{\ast}.   
\end{align}
This equation and~\eqref{s1barleqs1ast} lead to $h(\bar{s}_{1}, \bar{s}_{2}) = h(s_1^{\ast}, s_2^{\ast})$. Moreover, noting~\eqref{s1s2inBS} and the definition of ${\cal B}$ ensures $\min \{ h(s_{1}, s_{2}) \colon (s_{1}, s_{2}) \in {\cal S} \} = h(\bar{s}_{1}, \bar{s}_{2}) = h(s_{1}^{\ast}, s_{2}^{\ast})$. Hence, we have by~\eqref{sast_in_S} that $(s_1^{\ast}, s_2^{\ast}) \in {\cal B}$. Using $(\bar{s}_{1}, \bar{s}_{2}) \in {\cal A}$ and the definition of ${\cal A}$ imply
\begin{align} \label{s2barleqs2ast}
\frac{1}{2} \Vert \bar{s}_{2} \Vert^2 = g(\bar{s}_1, \bar{s}_2) \leq g(s_1^{\ast}, s_2^{\ast}) = \frac{1}{2} \Vert s_{2}^{\ast} \Vert^2.
\end{align}
Now, combining~\eqref{s1s2_in_calS} and~\eqref{eq:s1bar_s1ast} derives $\bar{s}_{2} \in {\cal S}_{2}$. It then follows from the second equality of~\eqref{eq:s1ast_s2ast} that $\frac{1}{2} \Vert s_{2}^{\ast} \Vert^2 \leq \frac{1}{2} \Vert \bar{s}_{2} \Vert^2$. This inequality and~\eqref{s2barleqs2ast} leads to $\min \{ \frac{1}{2} \Vert s_{2} \Vert^2 \colon s_{2} \in {\cal S}_{2} \} = \frac{1}{2} \Vert s_{2}^{\ast} \Vert^{2} = \frac{1}{2} \Vert \bar{s}_{2} \Vert^{2}$. Hence, by noting the second equality of~\eqref{eq:s1ast_s2ast}, we obtain
\begin{align} \label{eq:s2bar_s2ast}
\bar{s}_{2} \in \argmin \left\{ \frac{1}{2} \Vert s_{2} \Vert^{2} \colon s_{2} \in {\cal S}_{2} \right\}, \quad {\rm i.e.,} \quad \bar{s}_{2} = s_{2}^{\ast}.   
\end{align}
Therefore, according to~\eqref{eq:s1bar_s1ast} and~\eqref{eq:s2bar_s2ast}, we have the desired results.
\par
Second, we will verify that $(s_1^{\ast}, s_2^{\ast}) \in {\cal A}$. Let $(\widehat{s}_1, \widehat{s}_2) \in {\cal B}$ be arbitrary. Since $(s_1^{\ast}, s_2^{\ast}) \in {\cal B}$ is shown in the first part, we obtain $\frac{1}{2} \| \widehat{s}_1 \|^2 = h(\widehat{s}_1, \widehat{s}_2) = h(s_1^{\ast}, s_2^{\ast}) = \frac{1}{2} \| s_1^{\ast} \|^2$. Meanwhile, it is clear that $(\widehat{s}_1, \widehat{s}_2) \in {\cal S}$, that is, $\widehat{s}_1 \in {\cal S}_1$. It then follows from the first equality of~\eqref{eq:s1ast_s2ast} that $\widehat{s}_1 = s_1^{\ast}$. Moreover, again noting $(\widehat{s}_1, \widehat{s}_2) \in {\cal S}$ implies $\widehat{s}_2 \in {\cal S}_2$. Thus, we have from the second equality of~\eqref{eq:s1ast_s2ast} that $g(s_1^{\ast}, s_2^{\ast}) = \frac{1}{2} \| s_2^{\ast} \|^2 \leq \frac{1}{2} \| \widehat{s}_2 \|^2 = g(\widehat{s}_1, \widehat{s}_2)$. This inequality means $(s_1^{\ast}, s_2^{\ast}) \in {\cal A}$.
\end{proof}

\begin{lemma}\label{lem:s_converge}
If Assumptions~{\rm \ref{twoprob_solvable}},~{\rm \ref{ass:Adomf_closed}}, and~{\rm \ref{asm:sigma_frac_zero}} are satisfied, then $\{ \Vert \bar{s}_{1} (\sigma^{(k)}) \Vert \}$ and $\{ \Vert \bar{s}_{2} (\sigma^{(k)}) \Vert \}$ are convergent sequences, respectively.
\end{lemma}
\begin{proof}
First, we show that $\{ \| \bar{s}_{1} (\sigma^{(k)}) \| \}$ is a monotonically nonincreasing sequence and that $\{ \| \bar{s}_{2} (\sigma^{(k)}) \| \}$ is a monotonically nondecreasing sequence. Let $k \in \mathbb{N}$. Since $(\bar{x}(\sigma^{(k-1)}), \bar{s}_1(\sigma^{(k-1)}), \bar{s}_2(\sigma^{(k-1)}))$ is an optimal solution of problem (R$_{\sigma^{(k-1)}}$), it satisfies $\bar{x}(\sigma^{(k-1)}) \in \dom f$, $\bar{s}_2(\sigma^{(k-1)}) = b_2 - A_2 \bar{x}(\sigma^{(k-1)})$, and $\bar{s}_1(\sigma^{(k-1)}) = b_1 - A_1 \bar{x}(\sigma^{(k-1)})$. Hence, it is feasible for problem (R$_{\sigma^{(k)}}$), and therefore,
\begin{align*}
\frac{\sigma_1^{(k)}}{2} \|\bar{s}_1(\sigma^{(k)}) \|^2 + \frac{\sigma_2^{(k)}}{2} \|\bar{s}_2(\sigma^{(k)}) \|^2 
\leq 
\frac{\sigma_1^{(k)}}{2} \|\bar{s}_1(\sigma^{(k-1)}) \|^2 + \frac{\sigma_2^{(k)}}{2} \|\bar{s}_2(\sigma^{(k-1)}) \|^2.
\end{align*}
Multiplying both sides by $2/\sigma_{2}^{(k)}$ derives
\begin{align}\label{eq:s_bar_monotone3}
\eta^{(k)} \|\bar{s}_1(\sigma^{(k)}) \|^2 + \|\bar{s}_2(\sigma^{(k)}) \|^2 
\leq 
\eta^{(k)} \|\bar{s}_1(\sigma^{(k-1)}) \|^2 + \|\bar{s}_2(\sigma^{(k-1)}) \|^2.
\end{align}
Similarly, for the optimal solution $(\bar{x}(\sigma^{(k)}), \bar{s}_1(\sigma^{(k)}), \bar{s}_2(\sigma^{(k)}))$ of problem (R$_{\sigma^{(k)}}$), we obtain
\begin{align}\label{eq:s_bar_monotone4}
\eta^{(k-1)} \|\bar{s}_1(\sigma^{(k-1)}) \|^2 + \|\bar{s}_2(\sigma^{(k-1)}) \|^2 
\leq 
\eta^{(k-1)} \|\bar{s}_1(\sigma^{(k)}) \|^2 + \|\bar{s}_2(\sigma^{(k)}) \|^2.
\end{align}
Adding \eqref{eq:s_bar_monotone3} and \eqref{eq:s_bar_monotone4} implies
\begin{align*}
\left( \eta^{(k)} - \eta^{(k-1)} \right) 
\left( \|\bar{s}_1(\sigma^{(k)})\|^2 - \|\bar{s}_1(\sigma^{(k-1)})\|^2 \right) \leq 0.
\end{align*}
By Assumption~\ref{asm:sigma_frac_zero}, we have $\eta^{(k)} - \eta^{(k-1)} \geq 0$. If $\eta^{(k)} - \eta^{(k-1)} > 0$, then $\|\bar{s}_1(\sigma^{(k)})\| \leq \|\bar{s}_1(\sigma^{(k-1)})\|$ clearly holds. Meanwhile, if $\eta^{(k)} - \eta^{(k-1)} = 0$, problems~(${\cal R}_{\sigma^{(k)}}$) and~(${\cal R}_{\sigma^{(k-1)}}$) have the same unique optimizer. Moreover, item~(b) of Lemma~\ref{lem:equivalent} guarantees that $\bar{s}(\sigma^{(k)})$ and $\bar{s}(\sigma^{(k-1)})$ are the unique optimizers of~(${\cal R}_{\sigma^{(k)}}$) and~(${\cal R}_{\sigma^{(k-1)}}$), respectively. Hence, these facts yield
\begin{align}
\bar{s}_1(\sigma^{(k)}) = \bar{s}_1(\sigma^{(k-1)}), \quad \bar{s}_2(\sigma^{(k)}) = \bar{s}_2(\sigma^{(k-2)}). \label{eq:unique_opt_same}
\end{align}
As a result, it can be verified that $\|\bar{s}_1(\sigma^{(k)})\| \leq \|\bar{s}_1(\sigma^{(k-1)})\|$. Moreover, dividing both sides of \eqref{eq:s_bar_monotone3} by $\eta^{(k)}$ and both sides of \eqref{eq:s_bar_monotone4} by $\eta^{(k-1)}$, and then adding them, we obtain
\begin{align*}
\left( \frac{1}{\eta^{(k)}} - \frac{1}{\eta^{(k-1)}} \right) 
\left( \|\bar{s}_2(\sigma^{(k)})\|^2 - \|\bar{s}_2(\sigma^{(k-1)})\|^2 \right) \leq 0.
\end{align*}
Again using Assumption~\ref{asm:sigma_frac_zero} derives $1/\eta^{(k)} - 1/\eta^{(k-1)} \leq 0$. If $1/\eta^{(k)} - 1/\eta^{(k-1)} < 0$, we readily have $\|\bar{s}_2(\sigma^{(k-1)})\| \leq \|\bar{s}_2(\sigma^{(k)})\|$. If $1/\eta^{(k)} - 1/\eta^{(k-1)} = 0$, the same arguments described above leads to~\eqref{eq:unique_opt_same}. Therefore, both the cases ensures
$\|\bar{s}_2(\sigma^{(k-1)})\| \le \|\bar{s}_2(\sigma^{(k)})\|$.
\par
Next, we show that $\{ \| \bar{s}_{1} (\sigma^{(k)}) \| \}$ and $\{ \| \bar{s}_{2} (\sigma^{(k)}) \| \}$ are convergent. Since $\{ \| \bar{s}_{1} (\sigma^{(k)}) \| \}$ is a monotonically nonincreasing sequence bounded below, it is convergent. Thus, it suffices to show that $\{ \| \bar{s}_{2} (\sigma^{(k)}) \| \}$ is convergent. By Lemma~\ref{lem:true_sol_unique}, let $s_{2}^{\ast} \in \R^{m_2}$ be an optimal solution to 
$\min \{ \frac{1}{2} \| s_{2} \|^2 \colon s_{2} \in {\cal S}_{2} \}$.
Then $s_{2}^{\ast} \in {\cal S}_{2}$ holds. That is, there exists $x^{\ast} \in \dom f$ such that $s_{1}^{\ast} = b_1 - A_1 x^{\ast}$ and $s_{2}^{\ast} = b_2 - A_2 x^{\ast}$. Since $(x^{\ast}, s_{1}^{\ast}, s_{2}^{\ast})$ is feasible for problem (R$_{\sigma^{(k)}}$), we obtain
\begin{align}\label{eq:opt_feasible}
\frac{\sigma_1^{(k)}}{2} \|\bar{s}_1(\sigma^{(k)}) \|^2 + \frac{\sigma_2^{(k)}}{2} \|\bar{s}_2(\sigma^{(k)}) \|^2 
\leq 
\frac{\sigma_1^{(k)}}{2} \| s_1^{\ast} \|^2 + \frac{\sigma_2^{(k)}}{2} \| s_2^{\ast} \|^2.
\end{align}
Note that $(\bar{s}_1(\sigma^{(k)}), \bar{s}_2(\sigma^{(k)})) \in {\cal S}$. Moreover, Lemma~\ref{lem:caA_calB} ensures $(s_1^{\ast}, s_2^{\ast}) \in {\cal B}$, that is, $\frac{1}{2} \| s_1^{\ast} \|^2 = h(s_1^{\ast}, s_2^{\ast}) \leq h(\bar{s}_1(\sigma^{(k)}), \bar{s}_2(\sigma^{(k)})) = \frac{1}{2} \| \bar{s}_1(\sigma^{(k)}) \|^2$. It then follows from~\eqref{eq:opt_feasible} that $\|\bar{s}_2(\sigma^{(k)}) \| \leq \| s_2^{\ast} \|$. Hence, $\{ \| \bar{s}_2(\sigma^{(k)}) \| \}$ is a bounded and monotonically nondecreasing sequence, and hence it is convergent.
\end{proof}

\begin{theorem} \label{thm:sbar_sast}
Suppose that Assumptions~{\rm \ref{twoprob_solvable}},~{\rm \ref{ass:Adomf_closed}}, and~{\rm \ref{asm:sigma_frac_zero}} are satisfied. Then, the sequences $\{ \bar{s}_1(\sigma^{(k)}) \}$ and $\{ \bar{s}_2(\sigma^{(k)}) \}$ satisfy
\begin{align*}
\lim_{k \to \infty} \bar{s}_1(\sigma^{(k)}) = s_1^{\ast}, \quad \lim_{k \to \infty} \bar{s}_2(\sigma^{(k)}) = s_2^{\ast}.
\end{align*}
\end{theorem}
\begin{proof}
From Lemma~\ref{lem:s_converge}, the sequence $\{ \bar{s}(\sigma^{(k)}) = (\bar{s}_1(\sigma^{(k)}), \bar{s}_2(\sigma^{(k)}) ) \}$ has an accumulation point. Let $\widehat{s} \coloneqq (\widehat{s}_1, \widehat{s}_2) \in \R^{m_1} \times \R^{m_2}$ be an arbitrary accumulation point of $\{ \bar{s}(\sigma^{(k)}) \}$. Then, there exists a subsequence $\{ \bar{s}(\sigma^{(k_j)}) \} \subset \{ \bar{s}(\sigma^{(k)}) \}$ such that $\bar{s}(\sigma^{(k_j)}) \to \widehat{s}$ as $j \to \infty$. Let $j \in \mathbb{N} \cup \{ 0 \}$ be arbitrary. 
Lemma~\ref{lem:equivalent} ensures that
\begin{align} \label{widehat_s_in_S}
(\bar{s}_{1}(\sigma^{(k_j)}), \bar{s}_{2}(\sigma^{(k_j)})) \in \argmin \left\{ \frac{\sigma_{1}^{(k_j)}}{2} \Vert s_1 \Vert^2 + \frac{\sigma_2^{(k_j)}}{2} \Vert s_2 \Vert^2 \colon (s_1, s_2) \in {\cal S} \right\}.
\end{align}
% Recall that $\bar{s}(\sigma^{(k_j)})$ is a unique optimizer of~(${\cal R}_{\sigma^{(k_j)}}$). Hence, we obtain
% \begin{align} \label{widehat_s_in_S}
% (\bar{s}_{1}(\sigma^{(k_j)}), \bar{s}_{2}(\sigma^{(k_j)})) = \argmin \left\{ \frac{\sigma_{1}^{(k_j)}}{2} \Vert s_1 \Vert^2 + \frac{\sigma_2^{(k_j)}}{2} \Vert s_2 \Vert^2 \colon (s_1, s_2) \in {\cal S} \right\}.
% \end{align}
Notice that ${\cal S} = b - A(\dom f)$ is closed by Assumption~\ref{ass:Adomf_closed}, and that $(\bar{s}_{1}(\sigma^{(k_j)}), \bar{s}_{2}(\sigma^{(k_j)})) \in {\cal S}$ by~\eqref{widehat_s_in_S}. These facts imply that
\begin{align} \label{eq:s12_converge}
\lim_{j \to \infty} \bar{s}_{1}(\sigma^{(k_j)}) = \widehat{s}_{1}, \quad \lim_{j \to \infty} \bar{s}_{2}(\sigma^{(k_j)}) = \widehat{s}_{2}, \quad (\widehat{s}_{1}, \widehat{s}_{2}) \in {\cal S}.
\end{align}
Now, we show that $(\widehat{s}_{1}, \widehat{s}_{2}) \in {\cal B}$. Let $(s_1, s_2) \in {\cal S}$ be arbitrary. Noting~\eqref{widehat_s_in_S} yields
\begin{align} \label{ineq:s1s2_s1s2}
\frac{\sigma_1^{(k_j)}}{2} \Vert \bar{s}_{1}(\sigma^{(k_j)}) \Vert^2 + \frac{\sigma_2^{(k_j)}}{2} \Vert \bar{s}_{2}(\sigma^{(k_j)}) \Vert^2 \leq \frac{\sigma_1^{(k_j)}}{2} \Vert s_1 \Vert^2 + \frac{\sigma_2^{(k_j)}}{2} \Vert s_2 \Vert^2,
\end{align}
which leads to
\begin{align*}
\frac{1}{2} \Vert \bar{s}_{1}(\sigma^{(k_j)}) \Vert^2 + \frac{1}{2\eta^{(k_j)}} \Vert \bar{s}_{2}(\sigma^{(k_j)}) \Vert^2 \leq \frac{1}{2} \Vert s_1 \Vert^2 + \frac{1}{2\eta^{(k_j)}} \Vert s_2 \Vert^2.
\end{align*}
Since $1/\eta^{(k_j)} \to 0~(j \to \infty)$ from Assumption~\ref{asm:sigma_frac_zero}, it follows from~\eqref{eq:s12_converge} that $h(\widehat{s}_{1}, \widehat{s}_{2}) = \frac{1}{2} \Vert \widehat{s}_{1} \Vert^2 \leq \frac{1}{2} \Vert s_{1} \Vert^2 = h(s_1, s_2)$. Moreover, recall that $(\widehat{s}_1, \widehat{s}_2) \in {\cal S}$ is satisfied from~\eqref{eq:s12_converge}. Therefore, these facts and the definition of ${\cal B}$ derive $(\widehat{s}_{1}, \widehat{s}_{2}) \in {\cal B}$.
\par
Next, we prove the desired assertion. Let $(s_1, s_2) \in {\cal B}$. Recall that $(\bar{s}_{1}(\sigma^{(k_j)}), \bar{s}_{2}(\sigma^{(k_j)})) \in {\cal S}$ by~\eqref{widehat_s_in_S}. It then follows from $(s_1, s_2) \in {\cal B}$ and the definition of ${\cal B}$ that
\begin{align} \label{ineq:hs1s2}
\frac{1}{2} \Vert s_1 \Vert^2 = h(s_1, s_2) \leq h(\bar{s}_{1}(\sigma^{(k_j)}), \bar{s}_{2}(\sigma^{(k_j)})) = \frac{1}{2} \Vert \bar{s}_{1}(\sigma^{(k_j)}) \Vert^2    
\end{align}
Combining~\eqref{ineq:s1s2_s1s2} and~\eqref{ineq:hs1s2} yields $g(\bar{s}_{1}(\sigma^{(k_j)}),\bar{s}_{2}(\sigma^{(k_j)})) = \frac{1}{2} \Vert \bar{s}_{2}(\sigma^{(k_j)}) \Vert^2 \leq \frac{1}{2} \Vert s_2 \Vert^2 = g(s_1,s_2)$.
% \begin{align*} %\label{ineq:s1s2_s1s2}
% g(\bar{s}_{1}(\sigma^{(k_j)}),\bar{s}_{2}(\sigma^{(k_j)})) = \frac{1}{2} \Vert \bar{s}_{2}(\sigma^{(k_j)}) \Vert^2 \leq \frac{1}{2} \Vert s_2 \Vert^2 = g(s_1,s_2).
% %\frac{\sigma_1^{(k_j)}}{2} \Vert \bar{s}_{1}(\sigma^{(k_j)}) \Vert^2 + \frac{\sigma_2^{(k_j)}}{2} \Vert \bar{s}_{2}(\sigma^{(k_j)}) \Vert^2 \leq \frac{\sigma_1^{(k_j)}}{2} \Vert \bar{s}_{1}(\sigma^{(k_j)}) \Vert^2 + \frac{\sigma_2^{(k_j)}}{2} \Vert s_2 \Vert^2, \quad {\rm i.e.,} \quad \frac{1}{2} \Vert \bar{s}_{2}(\sigma^{(k)}) \Vert^2 \leq \frac{1}{2} \Vert s_2 \Vert^2.
% \end{align*}
Then, letting $j \to \infty$ and using~\eqref{eq:s12_converge} imply $g(\widehat{s}_{1}, \widehat{s}_{2}) \leq g(s_1, s_2)$. Since $(s_1, s_2) \in {\cal B}$ is arbitrary and $(\widehat{s}_{1}, \widehat{s}_{2}) \in {\cal B}$ holds, the definitions of ${\cal A}$ and ${\cal B}$ mean $(\widehat{s}_{1}, \widehat{s}_{2}) \in {\cal A}$. Then, Lemma~\ref{lem:caA_calB} ensures $(\widehat{s}_{1}, \widehat{s}_{2}) = (s_{1}^{\ast}, s_{2}^{\ast})$, that is, any accumulation point of $\{ \bar{s}(\sigma^{(k)}) \}$ is equal to $s^{\ast}$. Therefore, the sequences $\{ \bar{s}_1(\sigma^{(k)}) \}$ and $\{ \bar{s}_2(\sigma^{(k)}) \}$ converge to $s_1^{\ast}$ and $s_2^{\ast}$, respectively.
\end{proof}

\section{Augmented Lagrangian method and its global convergence}

From the discussion in Section~\ref{sec:app_opt_shift}, it has been shown that, if the sequences $\{ \sigma_1^{(k)} \}$ and $\{ \sigma_2^{(k)} \}$ are chosen so as to satisfy Assumption~\ref{asm:sigma_frac_zero}, then the approximate hierarchically optimal shift $\{ \bar{s}(\sigma^{(k)}) \}$ converges to the hierarchically optimal shift $s^{\ast}$. 
Therefore, the proposed method is expected to generate a sequence that converges to an optimal solution of~\eqref{p1_linear_slack_opt} if the augmented Lagrangian associated with the following optimization problem~(P$_{\bar{s}(\sigma)}$) with $\sigma = \sigma^{(k)}$ is minimized at each iteration $k \in \mathbb{N} \cup \{0\}$:
\begin{align}
\label{p1_linear_slack_approx_opt}
\begin{aligned}
& \mini_{x \in \R^{n}} && f(x)
\\
& \subj && A_1 x - b_1 + \bar{s}_1(\sigma) = 0, ~ A_2 x - b_2 + \bar{s}_2(\sigma) = 0, ~ x \in \dom f.
\end{aligned}
\tag{$\mathrm{P}_{\bar{s}(\sigma)}$}
\end{align}

\noindent
The augmented Lagrangian ${\cal F}$ associated with problem~\eqref{p1_linear_slack_approx_opt} is defined as follows:
\begin{align}
\begin{aligned} \label{def:cal_F}
{\cal F}_{\rho, \sigma}(x, \lambda_1, \lambda_2) &\coloneqq f(x) + \lambda_1^\top (A_1 x - b_1 + \bar{s}_1(\sigma)) + \lambda_2^\top (A_2 x - b_2 + \bar{s}_2(\sigma)) 
\\
& \hspace{30mm} + \frac{\rho}{2} \Vert A_1 x - b_1 + \bar{s}_1(\sigma) \Vert^2 + \frac{\rho}{2} \Vert A_2 x - b_2 + \bar{s}_2(\sigma) \Vert^2.
\end{aligned}
\end{align}
Here, $\lambda_1 \in \R^{m_1}$ and $\lambda_2 \in \R^{m_2}$ denote the Lagrange multipliers associated with the equality constraints, and $\rho > 0$ is a penalty parameter.

\par
We now present an augmented Lagrangian method for problem~\eqref{p0_general}. The formal statement of the proposed method is given as Algorithm~\ref{alg_alm_linear}. It should be noted that the method considered here is equipped with the infeasibility control framework that enables automatic handling of prioritized constraint violations. Moreover, the method generates auxiliary Lagrange multipliers in addition to the Lagrange multipliers. These auxiliary multipliers are introduced to stably compute the numerical minimization of the augmented Lagrangian function. In particular, for each iteration $k \in \mathbb{N} \cup \{0\}$, the auxiliary multipliers are obtained by projecting the Lagrange multipliers $\lambda_1^{(k)}$ and $\lambda_2^{(k)}$ onto compact sets $\Gamma_1 \subset \R^{m_1}$ and $\Gamma_2 \subset \R^{m_2}$, respectively. Here, for a compact set ${\cal C} \subset \R^{m}$, we denote by ${\cal P}_{{\cal C}}(z)$ the projection of $z \in \R^{m}$ onto ${\cal C}$.

\begin{algorithm}[htb]
\caption{(Augmented Lagrangian method with the infeasibility control framework)} \label{alg_alm_linear}
\begin{algorithmic}[1]
\State Choose positive constants $\tau \in (0,1)$ and $\gamma > 1$. %, $\kappa_1 > 0$, and $\kappa_2 > 0$. 
Select compact sets $\Gamma_1 \subset \R^{m_1}$ and $\Gamma_2 \subset \R^{m_2}$. Set positive sequences $\{ \sigma_1^{(k)} \} \subset \R_{++}$, $\{ \sigma_2^{(k)} \} \subset \R_{++}$, and $\{ \varepsilon^{(k)} \} \subset \R_{++}$ satisfying Assumption~\ref{asm:sigma_frac_zero} and $\varepsilon^{(k)} \to 0~(k\to \infty)$. Initialize $s_1^{(0)} \in \R^{m_1}$, $s_2^{(0)} \in \R^{m_2}$, $\widehat{\lambda}_1^{(0)} \in \R^{m_1}$, $\widehat{\lambda}_2^{(0)} \in \R^{m_2}$, $u^{(0)} > 0$, and $\rho^{(0)} > 0$. Set $k=0$.
\Repeat
\State Solve problem~(R$_{\sigma^{(k)}}$) and find $( \bar{s}_1(\sigma^{(k)}), \bar{s}_2(\sigma^{(k)}) ) \in \R^{m_1} \times \R^{m_2}$. \Comment{Step~1}
\State Solve the following problem, and find its $\varepsilon^{(k)}$-optimal solution $x^{(k+1)} \in \R^{n}$. \Comment{Step~2}
\begin{align*}
\begin{aligned}
& \mini_{x \in \R^{n}} && {\cal F}_{\rho^{(k)}, \sigma^{(k)}} (x, \widehat{\lambda}_1^{(k)}, \widehat{\lambda}_2^{(k)})
\\
& \subj && x \in \dom f,
\end{aligned}
\end{align*}

\State Update the Lagrange multipliers and the other parameters as follows: \Comment{Step~3}
\begin{align*}
s_1^{(k+1)} &\coloneqq A_1 x^{(k+1)} - b_1 + \bar{s}_1(\sigma^{(k)}), & s_2^{(k+1)} &\coloneqq A_2 x^{(k+1)} - b_2 + \bar{s}_2(\sigma^{(k)}),
\\
\lambda_1^{(k+1)} &\coloneqq \widehat{\lambda}_1^{(k)} + \rho^{(k)} s_1^{(k+1)}, & \lambda_2^{(k+1)} &\coloneqq \widehat{\lambda}_2^{(k)} + \rho^{(k)} s_2^{(k+1)},
\\
\widehat{\lambda}_1^{(k+1)} &\coloneqq {\cal P}_{\Gamma_1}(\widehat{\lambda}_1^{(k)} + \rho^{(k)} s_1^{(k+1)}), & \widehat{\lambda}_2^{(k+1)} &\coloneqq {\cal P}_{\Gamma_2}(\widehat{\lambda}_2^{(k)} + \rho^{(k)} s_2^{(k+1)}),
\\
% \sigma_1^{(k+1)} &\coloneqq \kappa_1 \sigma_1^{(k)}, & \sigma_2^{(k+1)} &\coloneqq \kappa_2 \sigma_2^{(k)},
% \\
u^{(k+1)} &\coloneqq \Vert s_1^{(k+1)} \Vert + \Vert s_2^{(k+1)} \Vert,
&
\rho^{(k+1)} &\coloneqq \left\{
\begin{aligned}
& \rho^{(k)} && {\rm if} ~ u^{(k+1)} \leq \tau u^{(k)},
\\
& \gamma \rho^{(k)} && {\rm otherwise}.
\end{aligned}
\right.
\end{align*}

\State Set $k \gets k+1$. % and use the prescribed values $\sigma_1^{(k)}$, $\sigma_2^{(k)}$, and $\varepsilon^{(k)}$ in the next iteration.
\Until{the termination criterion is satisfied.}
\end{algorithmic}
\end{algorithm}

\begin{remark}
If problem~\eqref{p0_general} is feasible, then the optimal solution of problem~(R$_{\sigma^{(k)}}$) is given by $(\bar{x}(\sigma^{(k)}), 0,0)$, that is, $( \bar{s}_1(\sigma^{(k)}), \bar{s}_2(\sigma^{(k)}) ) = (0,0)$. In this case, Algorithm~\ref{alg_alm_linear} reduces to the standard augmented Lagrangian method.
\end{remark}

\par
Now, we define the function $\theta \colon \R^{m_1} \times \R^{m_2} \to \R \cup \{ + \infty \}$ as
\begin{align*}
\theta(\lambda_{1}, \lambda_{2}) \coloneqq - \inf_{x \in \dom f} {\cal L}(x, \lambda_{1}, \lambda_{2}),
\end{align*}
where ${\cal L} \colon \R^{n} \times \R^{m_1} \times \R^{m_2} \to \R \cup \{ + \infty \}$ is the Lagrange function of~\eqref{p0_general}, that is,
\begin{align*}
{\cal L}(x, \lambda_1, \lambda_2) \coloneqq f(x) + \lambda_1^\top (A_1 x - b_1) + \lambda_2^\top (A_2 x - b_2).
\end{align*}

\noindent
In what follows, we discuss the well-definedness of Algorithm~\ref{alg_alm_linear}. Recall that the solvability of~(R$_{\sigma^{(k)}}$) is verified by Proposition~\ref{pro:solvability_Rsigma}. Hence, if the minimization problem of the augmented Lagrangian can be solved at each iteration $k \in \mathbb{N} \cup \{ 0 \}$, Algorithm~\ref{alg_alm_linear} works well. Hence, we show its solvability. To this end, the property of $\theta$ is given.

\begin{lemma}\label{lem:dual_feasible}
Suppose that Assumptions~{\rm \ref{twoprob_solvable}} and~{\rm \ref{ass:sast_solvable}} are satisfied. Then, the function $\theta$ is closed, proper, and convex, and it satisfies
\begin{align*}
\theta(\lambda_1, \lambda_2) 
= f^{\ast}(- A_1^{\top} \lambda_1 - A_2^{\top} \lambda_2) + \lambda_1^{\top}b_1 + \lambda_2^{\top}b_2 \quad \forall (\lambda_1, \lambda_2) \in \dom \theta.
\end{align*}
\end{lemma}
\begin{proof}
Recall that problem~\eqref{p1_linear_slack_opt} is convex by item~(i) of Assumption~\ref{ass:sast_solvable}. Moreover, item~(ii) of Assumption~\ref{ass:sast_solvable} ensures that~\eqref{p1_linear_slack_opt} has an optimizer $\bar{x}$ and satisfies Slater's CQ. It then follows from~\cite[Proposition~5.3.3]{B2009} that the dual problem also has an optimizer $(\bar{\lambda}_1, \bar{\lambda}_2)$, and it satisfies $f(\bar{x}) = - \theta(\bar{\lambda}_1, \bar{\lambda}_2) + \bar{\lambda}_1^{\top} s_1^{\ast} + \bar{\lambda}_2^{\top} s_2^{\ast}$, that is, $\dom \theta \not = \emptyset$. 
\par
We take $(\lambda_1, \lambda_2) \in \dom \theta$ arbitrarily. According to the definitions of $\theta$ and its conjugate function, it can be verified that
\begin{align}
\theta(\lambda_1, \lambda_2) 
&= \sup_{x \in \dom f} \{ - f(x) - \lambda_1^{\top} (A_1 x - b_1) - \lambda_2^{\top}(A_2 x - b_2) \} \nonumber
\\
&= \sup_{x \in \dom f} \{ ( - A_1^{\top} \lambda_1 - A_2^{\top} \lambda_2 )^{\top}x - f(x) \} + \lambda_1^{\top}b_1 + \lambda_2^{\top}b_2 \nonumber
\\
&= f^{\ast}(- A_1^{\top} \lambda_1 - A_2^{\top} \lambda_2) + \lambda_1^{\top}b_1 + \lambda_2^{\top}b_2. \label{eq:theta_fast}
\end{align}
Now, we notice that $f^{\ast}$ is a closed proper convex function from item~(i) of Assumption~\ref{ass:sast_solvable} and~\cite[Corollary~12.2.1]{R1972}. It then follows from $\dom \theta \not = \emptyset$ and~\eqref{eq:theta_fast} that $\theta$ is a closed proper convex function.
\end{proof}

For each $\sigma \in \R_{++}^{2}$, we define the functions ${\cal L}_{\sigma} \colon \R^{n} \times  \R^{m_1} \times \R^{m_2} \to \R \cup \{ +\infty \}$, $\theta_{\sigma} \colon \R^{m_1} \times \R^{m_2} \to \R \cup \{ +\infty \}$, and $\nu_{\sigma} \colon \R^{m_1} \times \R^{m_2} \to \R \cup \{ +\infty \}$ as
\begin{align*}
{\cal L}_{\sigma}(x, \lambda_1, \lambda_2) 
&\coloneqq f(x) + \lambda_1^{\top}(A_1 x - b_1 + \bar{s}_1(\sigma)) + \lambda_2^{\top}(A_2 x - b_2 + \bar{s}_2(\sigma)),
\\
\theta_{\sigma}(\lambda_1, \lambda_2) 
&\coloneqq - \inf_{x \in \dom f} {\cal L}_{\sigma}(x, \lambda_1, \lambda_2),
\\
\nu_{\sigma}(s_1, s_2) 
&\coloneqq \min_{x \in \dom f} \{ f(x) \colon A_1 x - b_1 + \bar{s}_1(\sigma) + s_1 = 0, A_2 x - b_2 + \bar{s}_2(\sigma) + s_2 = 0 \},
\end{align*}
respectively, where we define $\nu_{\sigma}(s_1, s_2) = +\infty$ if the corresponding minimization problem is infeasible. Note that $\nu_{\sigma}$ is well-defined under Assumption~\ref{ass:ps_solvable}. Recall that $(\bar{x}(\sigma), \bar{s}_1(\sigma), \bar{s}_2(\sigma))$ is an optimal solution of~\eqref{p_sigma}, that is, $\nu_{\sigma}(0,0) \leq f(\bar{x}(\sigma))$. This implies that $\dom \nu_{\sigma} \not = \emptyset$. Moreover, from the definitions of ${\cal L}$ and $\theta$, it is clear that
\begin{align}
\theta_{\sigma}(\lambda_1, \lambda_2) 
% &= - \inf_{x \in \dom f} {\cal L}_{\sigma}(x, \lambda_1, \lambda_2) \nonumber
% \\
= - \inf_{x \in \dom f} {\cal L}(x, \lambda_1, \lambda_2) + \lambda_1^{\top}\bar{s}_1(\sigma) + \lambda_2^{\top} \bar{s}_2(\sigma)
= \theta(\lambda_1, \lambda_2) + \lambda_1^{\top}\bar{s}_1(\sigma) + \lambda_2^{\top} \bar{s}_2(\sigma) \label{eq:theta_sigma_theta}
\end{align}
for any $(\lambda_1, \lambda_2) \in \R^{m_1} \times \R^{m_2}$.

\begin{lemma} \label{lem:nu_theta}
If Assumptions~{\rm \ref{ass:ps_solvable}}--{\rm \ref{ass:Adomf_closed}} are satisfied, then $\theta_{\sigma}^{\ast} = \cl \nu_{\sigma}$ holds.
\end{lemma}
\begin{proof}
Let $\sigma \in \R_{++}^{2}$ be arbitrary. %By Lemma~\ref{lem:dual_feasible} and~\eqref{eq:theta_sigma_theta}, we obtain $\dom \theta_{\sigma} \not = \emptyset$. 
Let $(\lambda_1, \lambda_2) \in \R^{m_1} \times \R^{m_2}$ be arbitrary. Note that \begin{align*}
\nu_{\sigma}(s_1, s_2) = -\sup_{x \in \dom f} \{ - f(x) - \delta_{\{ 0 \}} (A_1 x - b_1 + \bar{s}_1(\sigma) + s_1) - \delta_{\{ 0 \}} (A_2 x - b_2 + \bar{s}_2(\sigma) + s_2) \}
\end{align*}
for all $(s_1, s_2) \in \dom \nu_{\sigma}$. Using this equality and the definitions of $\nu_{\sigma}$, $\nu_{\sigma}^{\ast}$, and $\theta_{\sigma}$ derives
\begin{align}
\nu_{\sigma}^{\ast}(\lambda_1, \lambda_2)
&= \sup_{(s_1, s_2) \in \dom \nu_{\sigma}} \{ \lambda_1^{\top} s_1 + \lambda_2^{\top} s_2 - \nu_{\sigma}(s_1, s_2) \} \nonumber
% \\
% &= \sup_{(s_1, s_2) \in \dom \nu} \left\{ \sup_{x \in \dom f} \{ \lambda_1^{\top}s_1 + \lambda_2^{\top}s_2 - f(x) - \delta_{\{ 0 \}}(A_1x - b_1 + \bar{s}_1(\sigma) + s_1) - \delta_{\{ 0 \}}(A_2x - b_2 + \bar{s}_2(\sigma) + s_2) \} \right\} \nonumber
\\
&= \sup_{x \in \dom f} \{ - f(x) - \lambda_1^{\top} (A_1 x - b_1 + \bar{s}_1(\sigma)) - \lambda_2^{\top} (A_2 x - b_2 + \bar{s}_2(\sigma)) \} \nonumber
\\
&= - \inf_{x \in \dom f} {\cal L}_{\sigma}(x, \lambda_1, \lambda_2) \nonumber
\\
&= \theta_{\sigma}(\lambda_1, \lambda_2). \label{eq:nuast_theta}
\end{align}
Meanwhile, Lemma~\ref{lem:dual_feasible} and~\eqref{eq:theta_sigma_theta} ensure that $\theta_{\sigma}$ is a closed proper convex function, and hence so is $\nu_{\sigma}^{\ast}$. Moreover, it is easy to see that $\nu_\sigma$ is proper and convex. Indeed, $\dom\nu_\sigma\neq\emptyset$, and the convexity follows from the convexity of
$f$ and the linearity of the constraints defining $\nu_\sigma$.
It then follows from~\eqref{eq:nuast_theta} and~\cite[Theorem~12.2]{R1972} that $\theta_{\sigma}^{\ast} = \nu_{\sigma}^{\ast\ast} = \cl \nu_{\sigma}$.
\end{proof}

\begin{proposition} \label{pro:welldef}
Suppose that Assumptions~{\rm \ref{ass:ps_solvable}}--{\rm \ref{ass:Adomf_closed}} are valid. Then, for any $\varepsilon > 0$, $\rho > 0$, $\sigma = (\sigma_1, \sigma_2) \in \R_{++}^2$, and $\lambda = (\lambda_1, \lambda_2) \in \R^{m_1} \times \R^{m_2}$, the problem of minimizing ${\cal F}_{\rho, \sigma}( \, \cdot \, , \lambda_1, \lambda_2)$ over $\dom f$ has at least one $\varepsilon$-optimal solution.
\end{proposition}
\begin{proof}
%Let $\varepsilon > 0$, $\rho > 0$, $\sigma = (\sigma_1, \sigma_2) \in \R_{++}^2$, and $\lambda = (\lambda_1, \lambda_2) \in \R^{m_1} \times \R^{m_2}$ be arbitrary. 
For simplicity, we define
\begin{align*}
\theta_{\sigma, \rho}(\lambda_1, \lambda_2) \coloneqq \inf \left\{ \theta_{\sigma}(\mu_1, \mu_2) + \frac{1}{2\rho} \Vert \lambda_1 - \mu_1 \Vert^{2} + \frac{1}{2\rho} \Vert \lambda_2 - \mu_2 \Vert^{2} \colon (\mu_1, \mu_2) \in \R^{m_1} \times \R^{m_2} \right\},
\end{align*}
where $\theta_{\sigma, \rho}$ is known as the Moreau envelope of $\theta_{\sigma}$ with the parameter $\rho$. Since $\theta_{\sigma}$ is a closed proper convex function according to Lemma~\ref{lem:dual_feasible} and~\eqref{eq:theta_sigma_theta}, we have by~\cite[Proposition~12.15]{BauschkeCombettes2010} that $\theta_{\sigma, \rho}(\lambda_1, \lambda_2)$ is finite, and the minimization problem of its right-hand side has a unique minimizer. Now, we consider the conjugate function of $r_{\rho, \lambda} \colon \R^{m_1} \times \R^{m_2} \to \R$, which is defined as
\begin{align*}
r_{\rho, \lambda}(\mu_1, \mu_2) &\coloneqq \frac{1}{2\rho} \Vert \lambda_1 - \mu_1 \Vert^{2} + \frac{1}{2\rho} \Vert \lambda_2 - \mu_2 \Vert^{2}.
\end{align*}
By simple calculation, we have
\begin{align}
r_{\rho, \lambda}^{\ast}(\mu_1, \mu_2) = \frac{\rho}{2} \Vert \mu_1 \Vert^2 + \frac{\rho}{2} \Vert \mu_2 \Vert^2 + \lambda_1^{\top} \mu_1 + \lambda_2^{\top} \mu_2. \label{eq:q_sig_lam}
\end{align}
Meanwhile, exploiting~\cite[Theorem~15.23 and Proposition~15.24]{BauschkeCombettes2010} and Lemma~\ref{lem:nu_theta} implies
\begin{align}
\theta_{\sigma, \rho}(\lambda_1, \lambda_2) 
= -\inf \{ \cl \nu_{\sigma}(\mu_1, \mu_2) + r^{\ast}_{\rho, \lambda}(-\mu_1, -\mu_2) \colon (\mu_1, \mu_2) \in \R^{m_1} \times \R^{m_2} \}. \label{eq:theta_sigma}
% \\
% &= -\inf \{ \nu_{\sigma}(-\mu_1, -\mu_2) + r^{\ast}_{\rho, \lambda}(\mu_1, \mu_2) \colon (\mu_1, \mu_2) \in \R^{m_1} \times \R^{m_2} \}. 
\end{align}
Now, let us denote $\alpha \coloneqq \inf \{ \cl \nu_{\sigma}(\mu_1, \mu_2) + r^{\ast}_{\rho, \lambda}(-\mu_1, -\mu_2) \colon (\mu_1, \mu_2) \in \R^{m_1} \times \R^{m_2} \}$ and $\beta \coloneqq \inf \{ \nu_{\sigma}(\mu_1, \mu_2) + r^{\ast}_{\rho, \lambda}(-\mu_1, -\mu_2) \colon (\mu_1, \mu_2) \in \R^{m_1} \times \R^{m_2} \}$, and we will show that $\alpha = \beta$.
% \begin{align}
% \theta_{\sigma, \rho}(\lambda_1, \lambda_2) = - \inf \{ \nu_{\sigma}(\mu_1, \mu_2) + r^{\ast}_{\rho, \lambda}(-\mu_1, -\mu_2) \colon (\mu_1, \mu_2) \in \R^{m_1} \times \R^{m_2} \}. \label{eq:alpha_beta}
% \end{align}
% For simplicity, we denote the right-hand sides of~\eqref{eq:theta_sigma} and~\eqref{eq:alpha_beta} by $\alpha$ and $\beta$, respectively. 
Let $\delta > 0$ be arbitrary. From the definition of $\alpha$, there exists $(\mu_1(\delta), \mu_2(\delta)) \in \R^{m_1} \times \R^{m_2}$ such that 
\begin{align}
\cl \nu_{\sigma}(\mu_1(\delta),\mu_2(\delta)) + r_{\rho, \lambda}^{\ast}(-\mu_1(\delta),-\mu_2(\delta)) < \alpha + \delta. \label{ineq:beta_delta}
\end{align}
Moreover, it can be easily seen that $(\mu_1(\delta), \mu_2(\delta), \cl \nu_{\sigma}(\mu_1(\delta),\mu_2(\delta))) \in \epi \cl \nu_{\sigma} = \cl \epi \nu_{\sigma}$, that is, there exists $\{ (\mu_1^{(j)}, \mu_2^{(j)}, \gamma^{(j)}) \} \subset \epi \nu_{\sigma}$ such that 
\begin{align}
\lim_{j \to \infty} \mu_1^{(j)} = \mu_1(\delta), \quad \lim_{j \to \infty} \mu_2^{(j)} = \mu_2(\delta), \quad \lim_{j \to \infty} \gamma^{(j)} = \cl \nu_{\sigma}(\mu_1(\delta),\mu_2(\delta)). \label{eq:three_lim}
\end{align}
By noting $(\mu_1^{(j)}, \mu_2^{(j)}, \gamma^{(j)}) \in \epi \nu_{\sigma} \subset \cl(\epi \nu_{\sigma}) = \epi \cl \nu_{\sigma}$ and the definition of $\beta$, we obtain $\beta \leq \nu_{\sigma}(\mu_1^{(j)}, \mu_2^{(j)}) + r_{\rho, \lambda}^{\ast}(-\mu_1^{(j)}, -\mu_2^{(j)}) \leq \gamma^{(j)} + r_{\rho, \lambda}^{\ast}(-\mu_1^{(j)}, -\mu_2^{(j)})$ for any $j \in \mathbb{N}$. Since~\eqref{eq:three_lim} is satisfied, taking $j \to \infty$ yields $\beta \leq \cl \nu_{\sigma}(\mu_1(\delta),\mu_2(\delta)) + r_{\rho, \lambda}^{\ast}(-\mu_1(\delta),-\mu_2(\delta))$. This fact and~\eqref{ineq:beta_delta} derive $\beta < \alpha + \delta$, that is, $\beta \leq \alpha$ because $\delta > 0$ is arbitrary. On the other hand, from $(\mu_1, \mu_2, \nu_{\sigma}(\mu_1,\mu_2)) \in \epi \nu_{\sigma} \subset \cl(\epi \nu_{\sigma}) = \epi \cl \nu_{\sigma}$, it is clear that $\cl \nu_{\sigma}(\mu_1, \mu_2) \leq \nu_{\sigma}(\mu_1, \mu_2)$ for any $(\mu_1, \mu_2) \in \R^{m_1} \times \R^{m_2}$. This fact implies that $\alpha \leq \beta$. As a result, we have $\alpha = \beta$. It then follows from~\eqref{eq:q_sig_lam} and~\eqref{eq:theta_sigma} that
\begin{align*}
\theta_{\sigma, \rho}(\lambda_1, \lambda_2) 
&= -\inf \{ \nu_{\sigma}(\mu_1, \mu_2) + r^{\ast}_{\rho, \lambda}(-\mu_1, -\mu_2) \colon (\mu_1, \mu_2) \in \R^{m_1} \times \R^{m_2} \}
\\
&= \sup \{ -\nu_{\sigma}(\mu_1, \mu_2) - r^{\ast}_{\rho, \lambda}(-\mu_1, -\mu_2) \colon (\mu_1, \mu_2) \in \R^{m_1} \times \R^{m_2} \}
\\
&= \sup \{ - f(x) - r^{\ast}_{\rho, \lambda}(A_1 x - b_1 + \bar{s}_1(\sigma), A_2 x - b_2 + \bar{s}_2(\sigma)) \colon x \in \dom f \}
\\
&= -\inf \{ {\cal F}_{\rho, \sigma} (x, \lambda_1, \lambda_2) \colon x \in \dom f \}.
\end{align*}
Since $\theta_{\sigma, \rho}(\lambda_1, \lambda_2)$ is finite and $\dom f$ is nonempty, the problem of minimizing ${\cal F}_{\rho, \sigma} ( \, \cdot \, , \lambda_1, \lambda_2)$ over $\dom f$ has at least one $\varepsilon$-optimal solution.
\end{proof}

\noindent
Proposition~\ref{pro:welldef} guarantees the well-definedness of Algorithm~\ref{alg_alm_linear}. Finally, we prove its global convergence.

\begin{theorem}
Suppose that Assumptions~{\rm \ref{ass:ps_solvable}--\ref{asm:sigma_frac_zero}} are satisfied. 
Then, any accumulation point of $\{ x^{(k)} \}$ generated by Algorithm~\ref{alg_alm_linear} is an optimal solution of~\eqref{p1_linear_slack_opt}.
\end{theorem}
\begin{proof}
Let $\omega \colon \R^{m_1} \times \R^{m_2} \to \R \cup \{ +\infty \}$ be defined by
\begin{align*}
\omega(s_1, s_2) \coloneqq \min \{ f(x) \colon x \in \dom f, A_1 x - b_1 + s_1 = 0, A_2 x - b_2 + s_2 = 0 \},
\end{align*}
where we define $\omega(s_1, s_2) = + \infty$ if the optimization problem in $\omega(s_1,s_2)$ is infeasible. Note that $\omega$ is a proper convex function. We show that $s^{\ast} \in {\rm ri}(\dom \omega)$. By the definition of $\omega$, we have $\dom \omega = {\cal S} = b - A(\dom f)$. It then follows from item~(ii) of Assumption~\ref{ass:sast_solvable} that
\begin{align} \label{s_in_ridomw}
s^{\ast} = b - A\bar{x} \in b - A({\rm ri}(\dom f)) = {\rm ri}(b - A(\dom f)) = {\rm ri}(\dom \omega),
\end{align}
where the second equality follows from~\cite[Propositions~1.3.6 and~1.3.7]{B2009}. Since $\omega$ is continuous on ${\rm ri}(\dom \omega)$ by~\cite[Proposition~1.3.11]{B2009}, combining Theorem~\ref{thm:sbar_sast} and~\eqref{s_in_ridomw} derives
\begin{align}
\lim_{k \to \infty} \omega(\bar{s}_1(\sigma^{(k)}), \bar{s}_2(\sigma^{(k)})) = \omega(s_1^{\ast}, s_2^{\ast}). \label{lim:omega}
\end{align}
Let $k \in \mathbb{N}$ be arbitrary. Because Assumption~\ref{ass:ps_solvable} guarantees the solvability of~(P$_{\bar{s}(\sigma^{(k)})}$), we denote by $y^{(k)}$ its minimizer. Recall that $x^{(k)}$ is an $\varepsilon^{(k-1)}$-optimal solution of $\inf \{ {\cal F}_{\rho^{(k-1)}, \sigma^{(k-1)}}(x, \widehat{\lambda}_1^{(k-1)}, \widehat{\lambda}_2^{(k-1)}) \colon x \in \dom f \}$. Then, we have
\begin{align} 
{\cal F}_{\rho^{(k-1)}, \sigma^{(k-1)}}(x^{(k)}, \widehat{\lambda}_1^{(k-1)}, \widehat{\lambda}_2^{(k-1)}) 
&\leq {\cal F}_{\rho^{(k-1)}, \sigma^{(k-1)}}(y^{(k-1)}, \widehat{\lambda}_1^{(k-1)}, \widehat{\lambda}_2^{(k-1)}) + \varepsilon^{(k-1)} \nonumber
\\
&= \omega(\bar{s}_1(\sigma^{(k-1)}), \bar{s}_2(\sigma^{(k-1)})) + \varepsilon^{(k-1)}, \label{ineq:cal_F_xz}
\end{align}
where the last equality is derived from $A_1 y^{(k-1)} - b_1 + \bar{s}_1(\sigma^{(k-1)}) = 0$, $A_2 y^{(k-1)} - b_2 + \bar{s}_2(\sigma^{(k-1)}) = 0$, and $f(y^{(k-1)}) = \omega(\bar{s}_1(\sigma^{(k-1)}), \bar{s}_2(\sigma^{(k-1)}))$. 
\par
Let $x^{\ast} \in \R^{n}$ be an arbitrary accumulation point of $\{ x^{(k)} \}$. Then, there exists ${\cal K} \subset \mathbb{N}$ such that $\{ x^{(k)} \}_{k \in {\cal K}}$ converges to $x^{\ast}$. From now on, we will show that
\begin{align}
A_1 x^{\ast} - b_1 + s_1^{\ast} = 0, \quad A_2 x^{\ast} - b_2 + s_2^{\ast} = 0. \label{eq:xast_feasible}
\end{align}
Now, Theorem~\ref{thm:sbar_sast} guarantees that $\{ \bar{s}_1(\sigma^{(k)}) \}$ and $\{ \bar{s}_2(\sigma^{(k)}) \}$ converge to $s_1^{\ast}$ and $s_2^{\ast}$, respectively. Hence, to prove~\eqref{eq:xast_feasible}, it suffices to verify that
\begin{align} \label{ineq:2norm}
\lim_{{\cal K} \ni k \to \infty} \left\| A_1 x^{(k)} - b_1 + \bar{s}_1(\sigma^{(k-1)}) \right\| = 0, \quad
\lim_{{\cal K} \ni k \to \infty} \left\| A_2 x^{(k)} - b_2 + \bar{s}_2(\sigma^{(k-1)}) \right\| = 0. 
\end{align}
Recall that $\{ \rho^{(k)} \}$ is a monotonically nondecreasing sequence. Then, we can consider two cases: (a) $\{ \rho^{(k)} \}$ is bounded; (b) $\{ \rho^{(k)} \}$ diverges.
\par
{\it Case} (a). Since $\{ \rho^{(k)} \}$ is bounded, its updating rule of Step~3 in Algorithm~\ref{alg_alm_linear} ensures there exists $\widehat{k} \in \mathbb{N}$ satisfying $\rho^{(k)} = \rho^{(\widehat{k})}$ for all $k \geq \widehat{k}$. Moreover, again using the updating rule of $\{ \rho^{(k)} \}$ yields $u^{(k)} \leq \tau u^{(k-1)} \leq \tau^2 u^{(k-2)} \leq \cdots \leq \tau^{k -\widehat{k}} u^{(\widehat{k})}$ for all $k \geq \widehat{k}$, that is, $u^{(k)} \to 0~(k \to \infty)$ because $\tau \in (0,1)$. It then follows from the definitions of $\{ s_1^{(k)} \}$, $\{ s_2^{(k)} \}$, and $\{ u^{(k)} \}$ that $\| A_1 x^{(k)} - b_1 + \bar{s}_1(\sigma^{(k-1)}) \| + \| A_2 x^{(k)} - b_2 + \bar{s}_2(\sigma^{(k-1)}) \| = u^{(k)} \to  0 ~ (k \to \infty)$. Thus, we can verify that~\eqref{ineq:2norm} holds in this case.
\par
{\it Case} (b). Dividing both sides of~\eqref{ineq:cal_F_xz} by~$\rho^{(k-1)}$ yields
\begin{align}
&\left\| A_1 x^{(k)} - b_1 + \bar{s}_1(\sigma^{(k-1)}) \right\|^2 + \left\| A_2 x^{(k)} - b_2 + \bar{s}_2(\sigma^{(k-1)}) \right\|^2 \nonumber
\\
&\leq \frac{2}{\rho^{(k-1)}} \omega(\bar{s}_1(\sigma^{(k-1)}), \bar{s}_2(\sigma^{(k-1)})) + \frac{2}{\rho^{(k-1)}} \varepsilon^{(k-1)} - \frac{2}{\rho^{(k-1)}} f(x^{(k)}) \nonumber
\\
& \qquad - \frac{2}{\rho^{(k-1)}} (\widehat{\lambda}_1^{(k-1)})^{\top} (A_1 x^{(k)} - b_1 + \bar{s}_1(\sigma^{(k-1)})) - \frac{2}{\rho^{(k-1)}} (\widehat{\lambda}_2^{(k-1)})^{\top} (A_2 x^{(k)} - b_2 + \bar{s}_2(\sigma^{(k-1)})). \label{ineq:feasibility_Psast}
\end{align}
By the boundedness of $\{ x^{(k)} \}_{k \in {\cal K}}$, there exists $\Gamma > 0$ such that $\| x^{(k)} \| \leq \Gamma$ for all $k \in {\cal K}$. Moreover, since the function $f$ is closed proper convex and $\bar{x} \in {\rm ri}(\dom f)$ holds thanks to Assumption~\ref{ass:sast_solvable}, it follows from~\cite[Theorem~23.4]{R1972} that $\alpha \in \partial f(\bar{x})$ exists, that is, $\alpha^{\top}x + \beta \leq f(x)$ for all $x \in \R^{n}$, where $\beta \coloneqq f(\bar{x}) - \alpha^{\top} \bar{x}$. These facts imply that $-f(x^{(k)}) \leq -\alpha^{\top} x^{(k)} - \beta \leq \sup \{ -\alpha^{\top} x - \beta \colon \| x \| \leq \Gamma \} < +\infty$ for all $k \in {\cal K}$. Meanwhile, it is clear that $\{ \widehat{\lambda}_1^{(k)} \}$, $\{ \widehat{\lambda}_2^{(k)} \}$, $\{ \bar{s}_1(\sigma^{(k)}) \}$, $\{ \bar{s}_2(\sigma^{(k)}) \}$, and $\{ \omega(\bar{s}_1(\sigma^{(k)}), \bar{s}_2(\sigma^{(k)})) \}$ are bounded. Then, taking ${\cal K} \ni k \to \infty$ in~\eqref{ineq:feasibility_Psast} implies~\eqref{ineq:2norm}.
\par
As a result, we conclude that~\eqref{ineq:2norm} is satisfied for all the cases, and hence~\eqref{eq:xast_feasible} also holds. Now, we recall that $\{ \widehat{\lambda}_1^{(k)} \}$ and $\{ \widehat{\lambda}_2^{(k)} \}$ are bounded and recall that $\{ \varepsilon^{(k)} \}$ converges to zero. By inequality~\eqref{ineq:cal_F_xz} and the definition of ${\cal F}_{\rho^{(k-1)}, \sigma^{(k-1)}}$, we have $f(x^{(k)}) + (\widehat{\lambda}_1^{(k-1)})^{\top} (A_1 x^{(k)} - b_1 + \bar{s}_1(\sigma^{(k-1)})) + (\widehat{\lambda}_2^{(k-1)})^{\top} (A_2 x^{(k)} - b_2 + \bar{s}_2(\sigma^{(k-1)})) \leq \omega(\bar{s}_1(\sigma^{(k-1)}), \bar{s}_2(\sigma^{(k-1)})) + \varepsilon^{(k-1)}$. Then, taking $\liminf_{k \to \infty}$ and exploiting~\eqref{lim:omega},~\eqref{ineq:2norm}, and the lower semicontinuity of $f$ lead to $f(x^{\ast}) \leq \omega(s_1^{\ast}, s_2^{\ast})$. This fact and~\eqref{eq:xast_feasible} mean that $x^{\ast}$ is an optimal solution of~\eqref{p1_linear_slack_opt}. 
\end{proof}

\section{Numerical experiments}
This section reports numerical results to demonstrate the validity of the convergence analysis of Algorithm~\ref{alg_alm_linear}. All experiments were carried out in MATLAB R2025a on a machine equipped with an Intel Core i9-9900K CPU (3.60 GHz). We consider a convex quadratic optimization problem arising from a network flow model. The problem is defined on a grid network consisting of $20 \times 20$ nodes, where each node is connected to its adjacent nodes (up, down, left, and right) via bidirectional edges. The network has $m=400$ nodes and $n=1520$ edges.
\par
We briefly describe the network flow problem. Let $A \in \R^{m \times n}$ denote the node-edge incidence matrix of the network, and let $b \in \R^{m}$ represent the supply-demand vector. In the experiments, all nodes on the top row are designated as supply nodes, whereas all nodes on the bottom row are designated as demand nodes. Accordingly, each component of $b$ is set to $1$ for demand nodes, $-1$ for supply nodes, and $0$ for intermediate nodes. The decision variable $x \in \R^{n}$ represents the flow on each edge, and the equation $Ax=b$ represents the flow-balance constraints at all nodes.
\par
To assign priorities to the constraints, we partition the rows of $A$ and $b$ according to the roles of the nodes. To this end, we express $A$ and $b$ as follows:
\begin{align*}
A = 
\begin{bmatrix}
a_1^{\top}
\\
\vdots
\\
a_m^{\top}
\end{bmatrix}
, \quad b = 
\begin{bmatrix}
[b]_1
\\
\vdots
\\
[b]_m
\end{bmatrix}
,
\end{align*}
where $a_{\ell} \in \R^{n}$ and $[b]_{\ell} \in \R$ for $\ell \in \{1, \ldots, m \}$.
The constraints corresponding to the demand nodes and the intermediate nodes are treated as higher-priority constraints, whereas those corresponding to the supply nodes are treated as lower-priority constraints. More precisely, let $\{ i_1, \ldots, i_p \}$ be the set of indices of the demand and intermediate nodes, and let $\{ j_1, \ldots, j_q \}$ be the set of indices of the supply nodes. Then, we define
\begin{align*}
A_1 \coloneqq
\begin{bmatrix}
a_{i_1}^{\top}
\\
\vdots
\\
a_{i_p}^{\top}
\end{bmatrix}
, \quad b_1 \coloneqq
\begin{bmatrix}
[b]_{i_1}
\\
\vdots
\\
[b]_{i_p}
\end{bmatrix}
, \quad A_2 \coloneqq
\begin{bmatrix}
a_{j_1}^{\top}
\\
\vdots
\\
a_{j_q}^{\top}
\end{bmatrix}
, \quad b_2 \coloneqq
\begin{bmatrix}
[b]_{j_1}
\\
\vdots
\\
[b]_{j_q}
\end{bmatrix}
.
\end{align*}
In this setting, the higher-priority constraints represent demand satisfaction and the flow-balance conditions at intermediate nodes, while the lower-priority constraints represent the supply conditions.
\par
To generate infeasible instances, the supply values in the lower-priority constraints are perturbed by replacing $b_2$ with $b_2 + \kappa e$, where $\kappa \geq 0$ is a parameter. Since the components of $b_2$ are equal to $-1$, this replacement changes each supply value from $-1$ to $-1+\kappa$. When $\kappa=0$, the total supply and demand are balanced, and the problem is feasible. When $\kappa>0$, the total available supply is reduced, and hence the full system of flow-balance constraints becomes infeasible. In this case, the proposed method is expected to preserve the higher-priority constraints as much as possible and to absorb the unavoidable inconsistency through the lower-priority supply constraints.
\par
The objective function is defined as a convex quadratic function with a positive semidefinite matrix $Q \in \R^{n \times n}$ and a vector $c \in \R^{n}$. Thus, the problem can be formulated as follows:
\begin{align}
\begin{aligned} \label{pro:network_flow}
& \mini_{x \in \R^{n}} && \frac{1}{2} x^{\top} Q x + c^{\top} x
\\
& \subj && A_1 x - b_1 = 0, ~ A_2 x - b_2 - \kappa e = 0.
\end{aligned} \tag{Q$_{\kappa}$}
\end{align}

\par
In this numerical experiment, we investigate the performance of Algorithm~\ref{alg_alm_linear} for both feasible and infeasible problems by solving~\eqref{pro:network_flow} with $\kappa = 0$ and $\kappa = \frac{1}{2}$, corresponding to the feasible and infeasible cases, respectively. For both problems, we set $Q \coloneqq I$ and $c \coloneqq \frac{1}{10} e$. The parameters and the compact sets were selected as follows:
\begin{gather} \label{param_init1}
\tau \coloneqq 0.1, ~~ \gamma \coloneqq 5, ~~ \Gamma_1 \coloneqq \{ v \in \R^{m_1} \colon -10^{6} e \leq v \leq 10^{6} e \}, ~~ \Gamma_2 \coloneqq \{ v \in \R^{m_2} \colon -10^{6} e \leq v \leq 10^{6} e \},
\end{gather} 
The initial values and the sequences $\{ \sigma_1^{(k)} \}$ and $\{ \sigma_2^{(k)} \}$ were selected as follows:
\begin{gather} 
\begin{gathered} \label{param_init2}
s_1^{(0)} = 0, ~~ s_2^{(0)} = 0, ~~ \widehat{\lambda}_1^{(0)} = 0, ~~ \widehat{\lambda}_2^{(0)} = 0, ~~ u^{(0)} \coloneqq 10^3, ~~ \rho^{(0)} \coloneqq 1,
\\
\sigma_1^{(k+1)} \coloneqq 10 \sigma_1^{(k)}, ~~ \sigma_1^{(0)} \coloneqq 1, ~~ \sigma_2^{(k+1)} \coloneqq 1.1 \sigma_2^{(k)}, ~~ \sigma_2^{(0)} \coloneqq 1. 
\end{gathered}
\end{gather}
In the implementation, the infeasibility control problem in Step~1 of Algorithm~\ref{alg_alm_linear} was solved by MATLAB's \texttt{quadprog} with the optimality tolerance set to $10^{-8}$. Since the optimality condition for the augmented Lagrangian subproblem in Step~2 reduces to a system of linear equations, this subproblem was solved by directly solving the resulting linear system. The termination criterion was set to $E^{(k)} \leq 10^{-6}$, where $E^{(k)}$ denotes the residual of the KKT conditions for problems~(P$_{\bar{s}(\sigma^{(k)})}$) evaluated at $(x^{(k)}, \lambda_1^{(k)}, \lambda_2^{(k)})$, that is,
\begin{gather*}
E^{(k)} \coloneqq \Vert Q x^{(k)} + c + A_1^{\top} \lambda_1^{(k)} + A_2^{\top} \lambda_2^{(k)} \Vert + \Vert A_1 x^{(k)} - b_1 + \bar{s}_1(\sigma^{(k)}) \Vert + \Vert A_2 x^{(k)} - b_2 + \bar{s}_2(\sigma^{(k)}) \Vert.
\end{gather*}
Moreover, for the infeasible problem, we compare the performance of Algorithm~\ref{alg_alm_linear} with that of the standard augmented Lagrangian (AL) method, which corresponds to the case in which $(\bar{s}_1(\sigma^{(k)}), \bar{s}_2(\sigma^{(k)})) = (0,0)$ is adopted in Step~1 of Algorithm~\ref{alg_alm_linear}. The standard AL method used the same values of $\tau$, $\gamma$, $\Gamma_1$, $\Gamma_2$, $s_1^{(0)}$, $s_2^{(0)}$, $\widehat{\lambda}_1^{(0)}$, $\widehat{\lambda}_2^{(0)}$, $u^{(0)}$, and $\rho^{(0)}$ as those used for Algorithm~\ref{alg_alm_linear}. The sequences $\{ \sigma_1^{(k)} \}$ and $\{ \sigma_2^{(k)} \}$ were not used in the standard AL method.
%The parameters and the initial points of the standard AL method were the same as those in~\eqref{param_init1} and~\eqref{param_init2} except for $\sigma_1^{(0)}$ and $\sigma_2^{(0)}$.
\par
Numerical results are reported in Tables~\ref{Results:feasible_Alg1}--\ref{Results:infeasible_AL}. To facilitate their interpretation, for each $k \in \mathbb{N}$, we define the following quantities:
\begin{gather*}
r_1^{(k)} \coloneqq \Vert \bar{s}_1(\sigma^{(k)}) - s_1^{\ast} \Vert, \quad r_2^{(k)} \coloneqq \Vert \bar{s}_2(\sigma^{(k)}) - s_2^{\ast} \Vert,
\end{gather*}
where $s_1^{\ast}$ and $s_2^{\ast}$ are defined by~\eqref{eq:true_solution} and~\eqref{eq:true_solution2}, respectively.
\par
Tables~\ref{Results:feasible_Alg1} and~\ref{Results:infeasible_Alg1} report the numerical results of Algorithm~\ref{alg_alm_linear} for the feasible and infeasible problems, respectively. In both cases, the residual measure $E^{(k)}$ decreases steadily and reaches the termination criterion within a small number of iterations. This indicates that the proposed method successfully solves the shifted problems generated by the infeasibility control framework. For the feasible problem, the hierarchically optimal shift $s^{\ast}$ is equal to the zero vector. Accordingly, the quantities $r_1^{(k)}$ and $r_2^{(k)}$ remain very small throughout the iterations, indicating that the approximate shift $\bar{s}(\sigma^{(k)})$ is essentially zero. The norms $\|s_1^{(k)}\|$ and $\|s_2^{(k)}\|$ also converge to zero, which is consistent with the feasibility of the original problem. For the infeasible problem, the original flow-balance constraints cannot be satisfied simultaneously because the available supply is reduced. Nevertheless, Table~\ref{Results:infeasible_Alg1} shows that $r_1^{(k)}$ and $r_2^{(k)}$ decrease, which means that the approximate shift $\bar{s}(\sigma^{(k)})$ approaches the hierarchically optimal shift $s^{\ast}$. Thus, the proposed method represents the unavoidable violation caused by the reduced supply through the shift variables, while keeping the violation of the high-priority demand and flow-balance constraints small. Moreover, the multiplier norms remain bounded, and the penalty parameter $\rho^{(k)}$ stabilizes after several iterations.
\par
In contrast, Table~\ref{Results:infeasible_AL} shows that the standard AL method behaves differently for the infeasible problem. Since no shift variables are introduced, the method attempts to enforce inconsistent constraints. As a result, the penalty parameter $\rho^{(k)}$ keeps increasing, and the multiplier norms grow rapidly. This behavior illustrates the advantage of the proposed infeasibility control framework over the standard AL method for infeasible problems with prioritized constraints.
\bigskip

\begin{table}[htbp]
\caption{Numerical results of Algorithm~\ref{alg_alm_linear} for the feasible problem} \label{Results:feasible_Alg1}
\vspace{-1mm}
\centering
\begin{tabular}{rcccccccccc}
\hline
$k$	&	$E^{(k)}$	&	$\| s_1^{(k)} \|$	&	$\| s_2^{(k)} \|$	&	$r_1^{(k)}$	&	$r_2^{(k)}$	&	$\rho^{(k)}$	&	$\| \lambda_1^{(k)} \|$	&	$\| \lambda_2^{(k)} \|$	\\
% $k$	&	$E^{(k)}$	&	$\| s_1^{(k)} \|$	&	$\| s_2^{(k)} \|$	&	$\| \bar{s}_1(\sigma^{(k)}) \|$	&	$\| \bar{s}_1(\sigma^{(k)}) \|$	&	$\rho^{(k)}$	\\
\hline
%0 & 1.28e+01 & 4.47e+00 & 4.47e+00 & 0.00e+00 & 0.00e+00 & 1.00e+00 & 0.00e+00 & 0.00e+00 \\
1 & 5.12e+00 & 2.89e+00 & 2.24e+00 & 2.69e$-$09 & 1.07e$-$09 & 1.00e+00 & 2.89e+00 & 2.24e+00 \\
2 & 4.04e+00 & 2.55e+00 & 1.49e+00 & 3.10e$-$10 & 6.88e$-$10 & 5.00e+00 & 5.38e+00 & 3.73e+00 \\
3 & 2.75e+00 & 1.94e+00 & 8.08e$-$01 & 4.78e$-$11 & 1.73e$-$10 & 2.50e+01 & 1.47e+01 & 7.77e+00 \\
4 & 1.17e+00 & 8.71e$-$01 & 3.02e$-$01 & 5.45e$-$12 & 2.07e$-$11 & 1.25e+02 & 3.60e+01 & 1.53e+01 \\
5 & 1.63e$-$01 & 1.22e$-$01 & 4.08e$-$02 & 5.54e$-$13 & 2.77e$-$12 & 6.25e+02 & 5.11e+01 & 2.04e+01 \\
6 & 5.12e$-$03 & 3.84e$-$03 & 1.28e$-$03 & 5.38e$-$14 & 1.61e$-$11 & 6.25e+02 & 5.35e+01 & 2.12e+01 \\
7 & 1.61e$-$04 & 1.21e$-$04 & 4.02e$-$05 & 5.77e$-$15 & 1.49e$-$10 & 6.25e+02 & 5.36e+01 & 2.12e+01 \\
8 & 5.07e$-$06 & 3.81e$-$06 & 1.26e$-$06 & 3.29e$-$15 & 1.09e$-$09 & 6.25e+02 & 5.36e+01 & 2.12e+01 \\
9 & 1.60e$-$07 & 1.20e$-$07 & 3.98e$-$08 & 3.03e$-$15 & 2.30e$-$08 & 6.25e+02 & 5.36e+01 & 2.12e+01 \\
\hline
\end{tabular}
\end{table}

\begin{table}[htbp]
\caption{Numerical results of Algorithm~\ref{alg_alm_linear} for the infeasible problem} \label{Results:infeasible_Alg1}
\vspace{-1mm}
\centering
\begin{tabular}{rcccccccccc}
\hline
$k$	&	$E^{(k)}$	&	$\| s_1^{(k)} \|$	&	$\| s_2^{(k)} \|$	&	$r_1^{(k)}$	&	$r_2^{(k)}$	&	$\rho^{(k)}$	&	$\| \lambda_1^{(k)} \|$	&	$\| \lambda_2^{(k)} \|$	\\
\hline
%0 & 1.06e+01 & 4.47e+00 & 2.24e+00 & 3.42e$-$15 & 2.24e+00 & 1.00e+00 & 0.00e+00 & 0.00e+00 \\
1 & 3.79e+00 & 2.56e+00 & 1.23e+00 & 4.87e$-$01 & 2.12e+00 & 1.00e+00 & 2.56e+00 & 1.23e+00 \\
2 & 3.36e+00 & 2.21e+00 & 1.15e+00 & 3.47e$-$01 & 1.51e+00 & 5.00e+00 & 4.71e+00 & 2.38e+00 \\
3 & 2.72e+00 & 1.85e+00 & 8.70e$-$01 & 9.59e$-$02 & 4.18e$-$01 & 2.50e+01 & 1.35e+01 & 6.73e+00 \\
4 & 1.22e+00 & 8.88e$-$01 & 3.28e$-$01 & 1.27e$-$02 & 5.52e$-$02 & 1.25e+02 & 3.53e+01 & 1.49e+01 \\
5 & 1.71e$-$01 & 1.27e$-$01 & 4.34e$-$02 & 1.42e$-$03 & 6.20e$-$03 & 6.25e+02 & 5.10e+01 & 2.04e+01 \\
6 & 5.43e$-$03 & 4.06e$-$03 & 1.37e$-$03 & 1.57e$-$04 & 6.84e$-$04 & 6.25e+02 & 5.35e+01 & 2.12e+01 \\
7 & 1.79e$-$04 & 1.34e$-$04 & 4.57e$-$05 & 1.73e$-$05 & 7.53e$-$05 & 6.25e+02 & 5.36e+01 & 2.12e+01 \\
8 & 6.59e$-$06 & 4.84e$-$06 & 1.75e$-$06 & 1.90e$-$06 & 8.28e$-$06 & 6.25e+02 & 5.36e+01 & 2.12e+01 \\
9 & 3.12e$-$07 & 2.23e$-$07 & 8.91e$-$08 & 2.09e$-$07 & 9.11e$-$07 & 6.25e+02 & 5.36e+01 & 2.12e+01 \\
\hline
\end{tabular}
\end{table}

\begin{table}[htbp]
\caption{Numerical results of the standard AL method for the infeasible problem} \label{Results:infeasible_AL}
\vspace{-1mm}
\centering
\begin{tabular}{rcccccccccc}
\hline
$k$	&	$\rho^{(k)}$	&	$\| \lambda_1^{(k)} \|$	&	$\| \lambda_2^{(k)} \|$	\\
\hline
%0 & 1.00e+00 & 0.00e+00 & 0.00e+00 & \\
1 & 1.00e+00 & 2.66e+00 & 1.12e+00 & \\
2 & 5.00e+00 & 4.82e+00 & 1.86e+00 & \\
3 & 2.50e+01 & 1.26e+01 & 3.86e+00 & \\
4 & 1.25e+02 & 3.30e+01 & 6.39e+00 & \\
5 & 6.25e+02 & 8.87e+01 & 3.83e+00 & \\
6 & 3.13e+03 & 3.86e+02 & 7.29e+01 & \\
7 & 1.56e+04 & 1.89e+03 & 4.19e+02 & \\
8 & 7.81e+04 & 9.41e+03 & 2.14e+03 & \\
9 & 3.91e+05 & 4.70e+04 & 1.08e+04 & \\
10 & 1.95e+06 & 2.35e+05 & 5.39e+04 & \\
11 & 9.77e+06 & 1.18e+06 & 2.70e+05 & \\
12 & 4.88e+07 & 5.88e+06 & 1.35e+06 & \\
13 & 2.44e+08 & 2.94e+07 & 6.74e+06 & \\
14 & 1.22e+09 & 1.37e+08 & 3.15e+07 & \\
15 & 6.10e+09 & 6.13e+08 & 1.41e+08 & \\
16 & 3.05e+10 & 2.99e+09 & 6.87e+08 & \\
17 & 1.53e+11 & 1.49e+10 & 3.42e+09 & \\
18 & 7.63e+11 & 7.44e+10 & 1.71e+10 & \\
19 & 3.81e+12 & 3.72e+11 & 8.53e+10 & \\
20 & 1.91e+13 & 1.86e+12 & 4.27e+11 & \\
\hline
\end{tabular}
\end{table}

\FloatBarrier

\section{Conclusion}
In this paper, we studied convex optimization problems with prioritized equality constraints, particularly focusing on the infeasible case. To handle the hierarchy among constraints, we introduced the concept of a hierarchically optimal shift and reformulated the original problem as a shifted problem equipped with this optimal perturbation. We proposed an augmented Lagrangian method that incorporates the infeasibility control framework, which computes an approximate hierarchically optimal shift at each iteration. By analyzing a parameterized auxiliary problem regarding the approximate shift, we established that the approximate shift converges to the hierarchically optimal shift under suitable parameter updates. Moreover, we proved that the proposed algorithm is well-defined and converges to an optimal solution of the shifted problem equipped with the hierarchically optimal shift. The numerical experiments showed that the proposed method performs effectively for both feasible and infeasible problems, and that the infeasibility framework appropriately reflects the priority structure of constraints, whereas standard augmented Lagrangian methods may fail in the infeasible case.
\par
Future work includes extending the proposed framework to inequality constraints and more general convex composite problems, as well as developing accelerated variants and investigating local convergence properties. Another important direction is to extend the proposed approach to other optimization methods, such as sequential quadratic programming methods and interior-point methods, in order to incorporate prioritized constraints in more general nonlinear optimization settings.

\bibliographystyle{plain}
\bibliography{reference}

@book{BauschkeCombettes2010,
  title={{Convex Analysis and Monotone Operator Theory in Hilbert Spaces}},
  author={Bauschke, H. H. and P. L. Combettes},
  year={2010},
  publisher={Springer}
}

@book{B2009,
  title={Convex optimization theory},
  author={Bertsekas, D. P.},
  year={2009},
  publisher={Athena Scientific}
}

@article{BCN2010,
  title={Infeasibility detection and SQP methods for nonlinear optimization},
  author={Byrd, R. H. and Curtis, F. E. and Nocedal, J.},
  journal={SIAM Journal on Optimization},
  volume={20},
  number={5},
  pages={2281--2299},
  year={2010},
  publisher={SIAM}
}

@article{BCW2014,
  title={A sequential quadratic optimization algorithm with rapid infeasibility detection},
  author={Burke, J. V. and Curtis, F. E. and Wang, H.},
  journal={SIAM Journal on Optimization},
  volume={24},
  number={2},
  pages={839--872},
  year={2014},
  publisher={SIAM}
}

@article{BDM+2011,
  title={Robust optimization for emergency logistics planning: Risk mitigation in humanitarian relief supply chains},
  author={Ben-Tal, A. and Do Chung, B. and Mandala, S. R. and Yao, T.},
  journal={Transportation Research Part B: Methodological},
  volume={45},
  number={8},
  pages={1177--1189},
  year={2011},
  publisher={Elsevier}
}

@article{Co2002,
  title={Hard-constrained inconsistent signal feasibility problems},
  author={Combettes, P. L. and Bondon, P.},
  journal={IEEE Transactions on Signal Processing},
  volume={47},
  number={9},
  pages={2460--2468},
  year={1999},
  publisher={IEEE}
}

@article{CG2016,
  title={How the augmented Lagrangian algorithm can deal with an infeasible convex quadratic optimization problem},
  author={Chiche, A. and Gilbert, J. C.},
  journal={Journal of Convex Analysis},
  volume={23},
  pages={425--459},
  number={2},
  year={2016}
}

@article{DLS2020,
  title={A primal-dual interior-point method capable of rapidly detecting infeasibility for nonlinear programs},
  author={Dai, Y. H. and Liu, X. W. and Sun, J.},
  journal={Journal of Industrial and Management Optimization},
  volume={16},
  number={2},
  pages={1009--1035},
  year={2020},
  publisher={AIMS}
}

@article{DZ2023,
  title={The augmented Lagrangian method can approximately solve convex optimization with least constraint violation},
  author={Dai, Y. H. and Zhang, L.},
  journal={Mathematical Programming},
  volume={200},
  pages={633--667},
  year={2023},
  publisher={Springer}
}

@book{Eh2005,
  title={Multicriteria optimization},
  author={Ehrgott, M.},
  year={2005},
  publisher={Springer}
}

@article{FNW2007,
  title={Gradient projection for sparse reconstruction: Application to compressed sensing and other inverse problems},
  author={Figueiredo, M. A. T. and Nowak, R. D. and Wright, S. J.},
  journal={IEEE Journal of Selected Topics in Signal Processing},
  volume={1},
  number={4},
  pages={586--597},
  year={2007},
  publisher={IEEE}
}

@article{YMC2009,
  title={Evacuation transportation planning under uncertainty: a robust optimization approach},
  author={Yao, T. and Mandala, S. R. and Chung, B. D.},
  journal={Networks and Spatial Economics},
  volume={9},
  pages={171--189},
  year={2009},
  publisher={Springer}
}

@article{VSF1999,
  title={Infeasibility handling in linear MPC subject to prioritized constraints},
  author={Vada, J. and Slupphaug, O. and Foss, B. A.},
  journal={IFAC Proceedings Volumes},
  volume={32},
  number={2},
  pages={6763--6768},
  year={1999},
  publisher={Elsevier}
}

@article{VSJ+2001,
  title={Linear MPC with optimal prioritized infeasibility handling: application, computational issues and stability},
  author={Vada, J. and Slupphaug, O. and Johansen, T. A. and Foss, B. A.},
  journal={Automatica},
  volume={37},
  number={11},
  pages={1835--1843},
  year={2001},
  publisher={Elsevier}
}

@book{Mi1999,
  title={Nonlinear multiobjective optimization},
  author={Miettinen, K.},
  volume={12},
  year={1999},
  publisher={Springer Science \& Business Media}
}

@book{R1972,
  title={Convex analysis},
  author={Rockafellar, R. T.},
  year={1972},
  publisher={Princeton University Press}
}

@article{TJR1998,
  title={Goal programming for decision making: An overview of the current state-of-the-art},
  author={Tamiz, M. and Jones, D. and Romero, C.},
  journal={European Journal of operational research},
  volume={111},
  number={3},
  pages={569--581},
  year={1998},
  publisher={Elsevier}
}

@article{VF2009,
  title={Probing the Pareto frontier for basis pursuit solutions},
  author={Van Den Berg, E. and Friedlander, M. P.},
  journal={SIAM Journal on Scientific Computing},
  volume={31},
  number={2},
  pages={890--912},
  year={2009},
  publisher={SIAM}
}
\end{document}